%% file: Upper_half_plane_paper_1_1030.tex
\documentclass[12pt,oneside]{amsart}
\RequirePackage[colorlinks,citecolor=blue,urlcolor=blue]{hyperref}

\usepackage{latexsym,amssymb,amsmath,amsfonts,amsthm}
\usepackage{amsthm,amsmath}
\usepackage{float}
\usepackage{enumerate}
\usepackage{epic}
\usepackage{fullpage}
\usepackage{bbm}
\usepackage{subfigure}
\usepackage{mathbbol}
\usepackage{graphicx}
\usepackage[toc,page]{appendix}
\usepackage{amsfonts}
\usepackage[all]{xy}
\usepackage{caption}
 \usepackage{geometry}                
\usepackage{graphicx}
\usepackage{amssymb}
\usepackage{epstopdf}
\usepackage{color}
\usepackage{bbm, dsfont}
\usepackage{hyperref}
\usepackage[utf8]{inputenc}
\usepackage{amssymb,amsthm,amsmath,amssymb,wrapfig,dsfont}
\usepackage[dvipsnames]{xcolor}
\definecolor{myred}{RGB}{251,154,133}
\definecolor{myblue}{RGB}{153,206,227}
\definecolor{mylightblue}{RGB}{0, 150, 255}
\definecolor{mygreen}{RGB}{32, 210, 64}
\definecolor{mygray}{RGB}{220, 220, 220}

\usepackage{tikz}
\usetikzlibrary{decorations.pathmorphing}
\tikzset{snake it/.style={decorate, decoration=snake}}
\usetikzlibrary{shapes.geometric,positioning,decorations.pathreplacing}

\newtheorem{theorem}{Theorem}
\newtheorem{definition}{Definition}
\newtheorem{lemma}{Lemma}[section]
\newtheorem{remark}{Remark}
\newtheorem{Proposition}{Proposition}
\newtheorem{corollary}{Corollary}

\def\beq{ \begin{equation} }
\def\eeq{ \end{equation} }

\def\ep{\epsilon}

\def\square{\vcenter{\vbox{\hrule height .4pt
  \hbox{\vrule width .4pt height 5pt \kern 5pt
        \vrule width .4pt} \hrule height .4pt}}}

\def\ZZ{\mathbb{Z}}

\begin{document}

\title{Stationary Harmonic Measure and DLA in the Upper half Plane}

\author{Eviatar B. Procaccia}
\address[Eviatar B. Procaccia\footnote{Research supported by NSF grant DMS-1407558}]{Texas A\&M University}
\urladdr{www.math.tamu.edu/~procaccia}
\email{eviatarp@gmail.com}
 
\author{Yuan Zhang}
\address[Yuan Zhang]{Texas A\&M University}
\urladdr{http://www.math.tamu.edu/~yzhang1988/}
\email{yzhang1988@math.tamu.edu}

\maketitle

\begin{abstract}
In this paper, we introduce the stationary harmonic measure in the upper half plane. By bounding this measure, we are able to define both the discrete and continuous time diffusion limit aggregation (DLA) in the upper half plane with absorbing boundary conditions. We prove that for the continuous model the growth rate is bounded from above by $o(t^{2+\ep})$. When time is discrete, we also prove a better upper bound of $o(n^{2/3+\ep})$, on the maximum height of the aggregate at time $n$. An important tool developed in this paper, is an interface growth process, bounding any process growing according to the stationary harmonic measure. Together with \cite{when_harmonic_measure_0} one obtains non zero growth rate for any such process. 
\end{abstract}

\tableofcontents

\section{Introduction}

In this paper, we consider the stationary harmonic measure in the upper half plane and the corresponding diffusion limit aggregation (DLA). The Diffusion Limited Aggregation (DLA) in $\ZZ^2$ was introduced in 1983 by Witten and Sander \cite{DLA_introduction} as a simple model to study the geometry and dynamics of physical systems governed by diffusive laws. The DLA is defined recursively as a process on subsets $A_n\in \ZZ^2$. Starting from $A_0=\{(0,0)\}$, at each time a new point $a_{n+1}$ sampled from the harmonic probability measure on $\partial^{out}A_n$ is added to $A_n$. Intuitively, $a_{n+1}$ is the first place that a random walk starting from infinity visits $\partial^{out}A_n$. 

Although DLA is straightforwardly defined and easily simulated on a computer, very little about it is known rigorously. One of the notable exceptions is shown by Kesten where a polynomial upper bound, which equal to $n^{2/3}$ when $d=2$ and $n^{2/d}$ when $d\ge 3$, of the growth rate on DLA arms is given, see Corollary in \cite{harmonic_measure_1987} or Theorem in \cite{DLA_long}. In a later work \cite{DLA_1991} Kesten improved the upper bound for DLA when $d=3$ to $\sqrt{n\log(n)}$. No non-trivial lower bounds have been proved till present day. It is in fact still open to rule out that the DLA converges to a ball, although numerical simulations clearly exclude this. 

Recently, this topic is re-visited by Benjamini and Yadin \cite{DLA_re-visit} where they ```clean up' Kesten's argument, and make it more robust". They proved upper bounds on the growth rates of DLA's on ``transitive graphs of polynomial growth, graphs of exponential growth, non-amenable graphs, super-critical percolation on $Z^d$ and high dimensional pre-Sierpinski carpets". 

In this paper, we further extend the reach of Kesten's idea to non-transitive graphs. We define the (horizontally) translation invariant stationary harmonic measure on the upper half plane with absorbing boundary condition and show the growth of such stationary harmonic measure in a connected subset intersecting $x-$axis is sub-linear with respect to the height. With the bounds found for our stationary harmonic measure, we will be able to define a continuous time DLA on the upper half plane and give upper bound on its growth rates.

We first define several sets and stopping times for our problem. Let $\mathcal{H}=\{(x,y)\in \ZZ^2, y\ge 0\}$ be the upper half plane (including $x$-axis), and $S_n, n\ge 0$ be a 2-dimensional simple random walk. For any $x\in \ZZ^2$, we will write 
$$
x=(x_1,x_2)
$$
with $x_i$ denote the $i$th coordinate of $x$, and $\|x\|=|x_1|+|x_2|$. Then let $L_n, D_n\subset \ZZ^2$ be defined as follows: for each nonnegative integer $n$, define
$$
L_n=\{(x,n), \ x\in \ZZ\},
$$
$$
V_n=\{(0,k), \ 0\le k\le n\}, 
$$
and 
$$
U_n=L_0\cup V_n.
$$
I.e., $L_n$ is the horizontal line of height $n$ while $U_n$ is $x-$axis plus the vertical line segment between $(0,0)$ and $(0,n)$. And we let $y_n=(0,n)$ be the ``end point" of $V_n$. Moreover, we use $\mathcal{P}_n\subset \mathcal{H}$ for an arbitrary finite path in the upper half plane connecting $y_n$ and the $x-$axis. One can immediately see that $V_n$ is one such path. 

And for each subset $A\subset \ZZ^2$ we define stopping times 
$$
\bar \tau_A=\min\{n\ge 0, \ S_n\in A \}
$$
and
$$
\tau_A=\min\{n\ge 1, \ S_n\in A \}.
$$
For any subsets $A_1\subset A_2$ and $B$ and any $y\in \ZZ^2$, by definition one can easily check that 
\beq
\label{basic 1}
\begin{aligned}
&P_y\left(\tau_{A_1}<\tau_{B}\right)\le P_y\left(\tau_{A_2}<\tau_{B}\right), \\
&P_y\left(\bar\tau_{A_1}<\bar\tau_{B}\right)\le P_y\left(\bar\tau_{A_2}<\bar\tau_{B}\right), 
\end{aligned}
\eeq
and that 
\beq
\label{basic 2}
\begin{aligned}
P_y\left(\tau_{B}<\tau_{A_2}\right)\le P_y\left(\tau_{B}<\tau_{A_1}\right),\\
P_y\left(\bar\tau_{B}<\bar\tau_{A_2}\right)\le P_y\left(\bar\tau_{B}<\bar\tau_{A_1}\right),
\end{aligned}
\eeq
where $P_y(\cdot)=P(\cdot|S_0=y)$.
Now we define the stationary harmonic measure on $\mathcal{H}$ which will serve as the Poisson intensity in our continuous time DLA model. For any connected $B\subset \mathcal{H}$, any edge $\vec e=x\to y$ with  $x\in B$, $y\in \mathcal{H}\setminus B$ and any $N$, we define 
\beq
\label{harmonic measure edge}
H_{B, N}(\vec e)=\sum_{z\in L_N\setminus B} P_z\left(S_{\bar \tau_{B\cup L_0}}=x, S_{\bar \tau_{B\cup L_0}-1}=y\right)
\eeq
By definition, a necessary condition for $H_{B, N}(\vec e)>0$ (although at this point we have not yet ruled out the possibility it equals to infinity) is $y\in \partial^{out}B$ and $|x-y|=1$. And for all $x\in B$, we can also define  
\beq
\label{harmonic measure point old}
H_{B, N}(x)=\sum_{\vec e \text{ starting from } x}H_{B, N}(\vec e) =\sum_{z\in L_N\setminus B} P_z\left(S_{\bar \tau_{B\cup L_0}}=x\right).
\eeq
And for each point $y\in \partial^{out}B$, we can also define 
\beq
\label{harmonic measure point new}
\hat H_{B, N}(y)=\sum_{\tiny\begin{aligned}\vec e &\text{ starting in $B$}\\ &\text{ ending at } y\end{aligned}}H_{B, N}(\vec e) =\sum_{z\in L_N\setminus B} P_z\left(\tau_{B}\le \tau_{L_0} , S_{\bar \tau_{B\cup L_0}-1}=y\right).
\eeq
By coupling and strong Markov property, we have $N\to H_{A,N}(e)$ is bounded and monotone in $N$. Thus
\begin{Proposition}
\label{proposition_well_define}
For any $B$ and $\vec e$ as above, there is a finite $H_{B}(\vec e)$ such that 
\beq
\lim_{N\to\infty} H_{B, N}(\vec e)=H_{B}(\vec e). 
\eeq
\end{Proposition}
And we call $H_{B}(\vec e)$ the stationary harmonic measure of $\vec e$ with respect to $B$. Thus we immediately have the limits $H_{B}(x)=\lim_{N\to\infty}H_{B, N}(x)$ and $\hat H_{B}(y)=\lim_{N\to\infty}\hat H_{B, N}(y)$ also exists and we call them the stationary harmonic measure of $x$ and $y$ with respect to $B$.  Although now we have the limit $H_{B}(x)$ exists, it can be zero everywhere for certain $B$. We do not need to worry about this when $B$ is finite. For each finite $B$, we let 
$$
H_{B}=\sum_{x\in B} H_{B}(x)=\sum_{y\in\partial^{out} B} \hat H_{B}(y)
$$
be the harmonic measure of $B$. Then we have $H_{B}$ is non-decreasing as $B$ gets larger:
\begin{Proposition}
\label{proposition_finite}
For any finite subsets $B_1\subset B_2\subset \mathcal{H}$,
$$
H_{B_2}\ge H_{B_1}. 
$$
\end{Proposition}

\begin{remark}
However, for infinite subset of $\mathcal{H}$, it is possible to have the harmonic measure equal to 0 everywhere. In fact, we prove that as long as $B$ has a linear spatial growth horizontally, $H_B(\cdot)$ is uniformly 0. On the other hand, we have also proved that for any $B$ with certain sub-linear spatial growth, it can have non-zero stationary harmonic measure. These results are presented in a separate paper \cite{when_harmonic_measure_0}. 
\end{remark}

After presenting the basic properties of our stationary harmonic measure, we can state our first main result which gives the following upper bounds on $H_{B, N}(x)$:
\begin{theorem}
\label{theorem: uniform_path}
There is some constant $C<\infty$ such that for each connected $B\subset \mathcal{H}$ with $L_0\subset B$ and each $x=(x_1,x_2)\in B\setminus L_0$, and any $N$ sufficiently larger than $x_2$
\beq
\label{uniform bound}
H_{B, N}(x)\le C x_2^{1/2}. 
\eeq
\end{theorem}

\begin{remark}
In this paper, we use $C$ and $c$ as constants in $(0,\infty)$ independent to the change of variables like $N$ or $n$. But their exact values can be different from place to place. 
\end{remark}

At the same time, we can also have the following result showing that for a point of height $n$, say $y_n$ without loss of generality, the harmonic measure is maximized (up to multiplying a constant) by $U_n$. I.e.

\begin{theorem}
\label{theorem: maximum_path}
There is some constant $c>0$ such that for all $N>n$,  
\beq
\label{maximum}
H_{U_n, N}(y_n)\ge c n^{1/2}. 
\eeq
\end{theorem}

With Proposition \ref{proposition_finite}, Theorem \ref{theorem: uniform_path} and \ref{theorem: maximum_path} and the bounds estimates in their proofs, we can further show that 

\begin{theorem}
\label{theorem: theorem: uniform_2}
There are constants $0<c,C<\infty$ such that for any finite and connected $B$ in $\mathcal{H}$,
\beq
\label{uniform upper bound_finite_total}
H_B\le C\left(\min_{x\in B}\{x_2\}+|B|\right),
\eeq
while 
\beq
\label{uniform lower bound_finite_total}
H_B\ge c \max_{x\in B}\{x_2\} \log^{-1}\left(\max_{x\in B}\{x_2\}\right), 
\eeq
\beq
\label{uniform lower bound_finite_total_1}
H_B\ge 1
\eeq
when $ \max_{x\in B}\{x_2\}=0$. 
\end{theorem}
And again, we also have the total harmonic measure is maximized (up to multiplying a constant) by the vertical line segment $V_n$ over all connected finite subsets with the same cardinality and intersecting $L_0$. 
\begin{theorem}
\label{theorem: theorem: maximum_2}
There is a constant $c>0$ such that for any $n$,
\beq
\label{maximum_finite_total}
H_{V_n}\ge cn. 
\eeq

\end{theorem}

With the stationary harmonic measure bounded in Theorem \ref{theorem: uniform_path}, we are now able to define our DLA in the upper half plane as a continuous time stochastic process $A_t, t\ge 0$ taking values on finite subsets of $\mathcal{H}$. First we have $A_0=\{0\}$. For each $t\ge 0$. $A_t$ grows at a Poisson rate of $H_{A_t}$ 
and add a new point on $\partial^{out}A_t$ according to the probability distribution 
$$
\tilde p(A_t,y)=\frac{\hat H_{A_t}(y)}{H_{A_t}}, \ y\in \mathcal{H}. 
$$
Similarly, we can also define the discrete DLA model $\{A_n\}_{n=0}^\infty$ in $\mathcal{H}$ which is the embedded Markov chain of $A_t$. I.e., at each $n$, $A_{n+1}=A_n\cup \{y\}$ where $y$ is sampled according to $\tilde p(A_n,y)$. 

First, by introducing a pure growth interacting particle system that dominates the continuous time process, we show that $A_t$ is well defined and estimate an upper bound on the growth rate of its arms. For any finite $A$ define 
$$
\|A\|=\max\{\|x\|, x\in A\}. 
$$ 
\begin{theorem}
\label{theorem: DLA_1}
$A_t$ is well defined on $t\in[0,\infty)$. And for any $\ep>0$, we have with probability one
\beq 
\label{growth_DLA_1}
\limsup_{t\to\infty}t^{-2-\ep}\|A_t\|=0. 
\eeq
\end{theorem}
Furthermore, we show that for any time $t$, $\|A_t\|$ has a finite $m$th order moment for all $m\ge 1$. 
\begin{theorem}
\label{theorem: moment}
For any integer $m\ge 1$ and any $t\ge 0$
\beq
\label{moment}
E\big[\|A_t\|^m\big]<\infty. 
\eeq
\end{theorem}

\begin{remark}
In our construction we are able to define the dominating interacting particle system starting from any initial configuration in $\{0,1\}^{\mathcal{H}}$, whose growth rate is given by the upper bound of the stationary harmonic measure found in Theorem \ref{theorem: uniform_path} . This, together with \cite{when_harmonic_measure_0}, allows us to define a horizontally translation invariant infinite DLA on $\mathcal{H}$ and estimate its (non-zero) growth rate. We call this the stationary DLA model, and it will be presented in \cite{Stationary_DLA}. We refer the reader to look at recent results on other stationary aggregation processes \cite{MR3619794,Berger-Kagan-Procaccia}.
\end{remark}

For the discrete time process let $h_n=\max_{x\in A_n}\{x_2\}$. By Theorem \ref{theorem: uniform_path} and \eqref{uniform lower bound_finite_total}, we see that the probability that a new point $y$ is added to the aggregation $A_n$ is no larger than $\log(h_n)/\sqrt{h_n}$. Then the Borel-Cantelli argument in Step (ii) of \cite{DLA_long} easily gives us a stronger upper bound on $h_n$:

\begin{theorem}
\label{theorem: DLA_2}
For any $\epsilon>0$, we have with probability one
$$
\limsup_{n\to\infty} n^{-\epsilon-2/3}h_n=0. 
$$
\end{theorem}

The structure of this paper is as follows: In Section \ref{section: basic_properties} we prove the more basic properties of the stationary harmonic measure, i.e., Proposition \ref{proposition_well_define} and \ref{proposition_finite}. Theorem \ref{theorem: maximum_path} is proved in Section \ref{section: uniform} by showing that certain arguments in \cite{harmonic_measure_1987} is actually sharp. Then we ``inverse" the argument for vertical line segment and prove Theorem \ref{theorem: maximum_path} in Section \ref{section: maximum}. In Section \ref{section: finite} we use the bounds found the the previous two sections and show Theorem \ref{theorem: theorem: uniform_2} and \ref{theorem: theorem: maximum_2} inductively. In Section \ref{section: particle system}, we use an interacting particle system argument to define the dominating process and prove Theorem \ref{theorem: DLA_1} and \ref{theorem: moment}. After that, Theorem \ref{theorem: DLA_2} follows immediately.

\section{Properties of stationary harmonic measure}
\label{section: basic_properties}

\subsection{Proof of Proposition \ref{proposition_well_define}}
To show Proposition \ref{proposition_well_define}, we first need to verify that the infinite summation defined in \eqref{harmonic measure point old} converges. Note that for $x_2>0$ and any $N>x_2$ and any $z\in L_N\setminus B$,
$$
P_z\left(S_{\bar \tau_{B}}=x\right)=\sum_{k=1}^\infty P_z\left(\bar \tau_{B\cup L_0}=k, S_k=x\right).
$$
And by time reversal and symmetry of simple random walk, we have 
\begin{align*}
P_z\left(\bar \tau_{B}=k, S_k=x\right)&=P_z\left(S_k=x, S_1,S_2,\cdots, S_{k-1}\notin B\cup L_0\right)\\
&=P_{x}\left(S_k=z, S_1,S_2,\cdots, S_{k-1}\notin B\cup L_0\right)\\
&=P_{x}\left(S_k=z, \tau_{B\cup L_0}>k\right).
\end{align*}
Thus
\begin{align*}
P_z\left(S_{\bar \tau_{B}}=x\right)&=\sum_{k=1}^\infty P_{x}\left(S_k=z, \tau_{B\cup L_0}>k\right)\\
&=E_{x}\Big[\text{number of visits to $z$ in time interval }[0,\tau_{B\cup L_0}) \Big].
\end{align*}
Then taking the summation over all $z\in L_N\setminus B$, we have 
\beq
\label{measure to expectation}
H_{B,N}(x)=E_{x}\Big[\text{number of visits to $L_N$ in time interval } [0,\tau_{B\cup L_0})\Big]. 
\eeq
Then noting that $\tau_{L_0}\ge \tau_{B \cup L_0}$, we have 
$$
H_{B,N}(x)\le E_{x}\Big[\text{number of visits to $L_N$ in time interval } [0,\tau_{L_0}]\Big].
$$
Moreover, for $N>x_2$, note that if we trace the jumps on the second coordinate of $S_n$, it gives an (embedded) 1-dimensional simple random walk. We can use the strong Markov property of random walk on stopping time $\bar \tau_{L_N}\wedge \tau_{L_0}$ 
$$
\begin{aligned}
E_{x}&\Big[\text{number of visits to $L_N$ in time interval } [0,\tau_{L_0}]\Big]\\
&=\sum_{w\in L_N}P_x(\tau_{L_N}< \tau_{L_0}, S_{\tau_{L_N}}=w) E_w\Big[\text{number of visits to $L_N$ in time interval} [0,\tau_{L_0}]\Big].
\end{aligned}
$$
Note that for each $w\in L_N$, 
\beq
\label{1_dimensional_mean}
E_w\Big[\text{number of visits to $L_N$ in time interval} [0,\tau_{L_0}]\Big]=\frac{4}{P_{w-(0,1)}\left(\tau_{L_0}<\tau_{L_N} \right)}=4N
\eeq
is actually independent to the choice of $w$, and that for all $N>x_2$
$$
P_x(\tau_{L_N}< \tau_{L_0})=\frac{x_2}{N}.
$$
We have
$$
\begin{aligned}
E_{x}&\Big[\text{number of visits to $L_N$ in time interval } [0,\tau_{L_0}]\Big]=4N\cdot P_x(\tau_{L_N}< \tau_{L_0}) =4x_2. 
\end{aligned}
$$
Thus we have shown that 
\beq
\label{weak bound}
H_{B,N}(x\to y)\le H_{B,N}(x)\le 4x_2<\infty. 
\eeq
Similarly, we can also show that for $x_2=0$, 
\beq
\label{weak bound 2}
H_{B,N}(x\to y)\le H_{B,N}(x)\le 1<\infty. 
\eeq
With $H_{B,N}(x\to y)$ uniformly bounded for all $N$, we next show that $H_{B,N}(x\to y)$ is monotonically decreasing with respect to $N$. I.e., for any $N>M>x_2+1$ we want to show that 
\beq
\label{monotone_harmonic}
H_{B,N}(x\to y)\le H_{B,M}(x\to y). 
\eeq
Recalling that 
$$
H_{B, N}(x\to y)=\sum_{z\in L_N\setminus B} P_z\left(S_{\bar \tau_{B\cup L_0}}=x, \ S_{\bar \tau_{B\cup L_0}-1}=y\right), 
$$
for each $N$ we can define $S^{(0,N)}_n$ be a simple random walk in some probability space $P(\cdot)$ starting at $(0,N)$, and $S^{(k,N)}_n=S^{(0,N)}_n+(k,0)$ for all $k\in \ZZ$. Noting that $S^{(k,N)}_n$ is a simple random walk starting at $(k,N)$, we have
\beq
H_{B, N}(x)=\sum_{k: \ (k,N)\in L_N\setminus B} P\left(S^{(k,N)}_{\bar \tau_{B\cup L_0}}=x, \ S^{(k,N)}_{\bar \tau_{B\cup L_0}-1}=y \right). 
\eeq
Recalling that $N>M>x_2$, a random walk starting at $L_N$ must first visit $L_M$ before it can ever reach $x$. Thus for stopping time 
$$
\bar\tau_{L_M}=\inf\left\{n: \ S^{(0,N)}_n\in L_M\right\}
$$
note that by definition we also have 
$$
\bar\tau_{L_M}=\inf\left\{n: \ S^{(k,N)}_n\in L_M\right\}
$$
and
$$
S^{(k,N)}_{\bar\tau_{L_M}}=k+S^{(0,N)}_{\bar\tau_{L_M}}
$$
for all $k\in \ZZ$. Thus by strong Markov property, we have for each $k$ such that $(k,N)\in L_N\setminus B$
$$
\begin{aligned}
&P\left(S^{(k,N)}_{\bar \tau_{B\cup L_0}}=x, \ S^{(k,N)}_{\bar \tau_{B\cup L_0}-1}=y \right)\\
&=\sum_{j\in Z}P\left(S^{0,N}_{\bar \tau_{L_M}}=(j,M), \ \bar \tau_{L_M}\le \bar \tau_{B-(k,0)} \right)P_{(j+k,M)}\left(S_{\bar \tau_{B\cup L_0}}=x, \ S_{\bar \tau_{B\cup L_0}-1}=y \right)\\
&\le \sum_{j\in Z}P\left(S^{0,N}_{\bar \tau_{L_M}}=(j,M)\right)P_{(j+k,M)}\left(S_{\bar \tau_{B\cup L_0}}=x, \ S_{\bar \tau_{B\cup L_0}-1}=y \right). 
\end{aligned}
$$
Taking summation over all $k$,
\beq
\label{harmonic_decreasing_1}
H_{B, N}(x)\le\sum_{j\in \ZZ} P\left(S^{(0,N)}_{\bar\tau_{L_M}}=(j,M)\right)\sum_{k: \ (k,N)\in L_N\setminus B} P_{(j+k,M)}\left(S_{\bar \tau_{B\cup L_0}}=x, \ S_{\bar \tau_{B\cup L_0}-1}=y \right). 
\eeq
Note that for any $(i,M)\in B$, $P_{(i,M)}\left(S_{\bar \tau_{B\cup L_0}}=x \right)=0$. Thus 
\beq
\label{harmonic_decreasing_2}
\begin{aligned}
&\sum_{k: \ (k,N)\in L_N\setminus B} P_{(j+k,M)}\left(S_{\bar \tau_{B\cup L_0}}=x, \ S_{\bar \tau_{B\cup L_0}-1}=y\right)\\
\le &\sum_{k: \ (k,M)\in L_M\setminus B} P_{(k,M)}\left(S_{\bar \tau_{B\cup L_0}}=x, \ S_{\bar \tau_{B\cup L_0}-1}=y \right)=H_{B,M}(x\to y). 
\end{aligned}
\eeq
Combining \eqref{harmonic_decreasing_1} and \eqref{harmonic_decreasing_2} we have \eqref{monotone_harmonic}. The fact that any monotonically decreasing nonnegative sequence is convergent finishes the proof of Proposition  \ref{proposition_well_define}. \qed

\subsection{Proof of Proposition \ref{proposition_finite}}

To show Proposition \ref{proposition_finite} for finite subsets, recalling the definition and the fact that both $B_1$ and $B_2$ are finite. for any sufficiently large $N$ such that $L_N\cap B_2=\O$, we have 
$$
H_{B_1}=\sum_{x\in B_1}\sum_{z\in L_N} P_z(S_{\bar \tau_{B_1\cup L_0}}=x)
$$
and 
$$
H_{B_2}=\sum_{x\in B_2}\sum_{z\in L_N} P_z(S_{\bar \tau_{B_2\cup L_0}}=x). 
$$
Changing the order of both summations we have 
$$
H_{B_1}=\sum_{z\in L_N}\sum_{x\in B_1} P_z(S_{\bar \tau_{B_1\cup L_0}}=x)=\sum_{z\in L_N}P_z(S_{\bar \tau_{B_1\cup L_0}}\in B_1)=\sum_{z\in L_N}P_z(\tau_{B_1}\le \tau_{L_0})
$$
which is smaller than or equal to 
$$
H_{B_2}=\sum_{z\in L_N}\sum_{x\in B_2} P_z(S_{\bar \tau_{B_2\cup L_0}}=x)=\sum_{z\in L_N}P_z(S_{\bar \tau_{B_2\cup L_0}}\in B_2)=\sum_{z\in L_N}P_z(\tau_{B_2}\le \tau_{L_0})
$$
by \eqref{basic 1}. \qed

\section{Uniform upper bounds on harmonic measure}
\label{section: uniform}

In this section, we improve the linear bound in \eqref{weak bound} to Theorem \ref{theorem: uniform_path}. Without loss of generality we can assume $x_2=n$. According to the definition of $H_{B,N}(x)$ and \eqref{basic 2}, we first note that for any $B'\subset B$, with $x\in B'$ and $L_0\subset B'$,
$$
H_{B,N}(x)\le H_{B',N}(x).
$$
And since $B$ is connected and $L_0\subset B$. There must be a finite nearest neighbor path 
$$
\mathcal{P}_n=\{x=P_0,P_1,P_2,\cdots, P_{k_n}\in L_0\}
$$
connecting $x$ and $L_0$, where $|P_i-P_{i+1}|=1$. And since $d(x,L_0)=n$, $|x-P_{k_n}|\ge n$. Define
$$
m_n=\inf\{i: |P_i-x|\ge n\}
$$
and 
$$
\mathcal{Q}_n=\{P_0,P_1,P_2,\cdots, P_{m_n}\}.
$$
One can immediately see that 
$$
\mathcal{Q}_n\subset B(x,2n). 
$$
Then for $B_n=L_0\cup \mathcal{Q}_n$, to prove Theorem \ref{theorem: uniform_path}, it suffices to show that 
\beq
\label{truncated path}
H_{B_n,N}(x)\le C n^{1/2}. 
\eeq
And since simple random walk is translation invariant, we can without loss of generality assume that $x_1=0$. To show \eqref{truncated path}, we first prove that the inequalities in Lemma 3 and 4 and Inequality (2.15) in \cite{harmonic_measure_1987} are actually asymptotically sharp. 

\subsection{Asymptotic sharpness lemmas}

In this subsection, we will temporally move back to $\ZZ^2$ rather than the upper half plane $\mathcal{H}$. The connection will be shown when we conclude the proof of Theorem  \ref{theorem: maximum_path}. The ``inverses" of both lemmas starts with similar arguments as in their original proof. While the inverse of Lemma 3 is on a more ``natural" direction and its proof is more or less the same, that of Lemma 4 is a more delicate and requires a slightly stronger condition. Once we have the two inverse lemmas, the asymptotic sharpness of (2.15) follows from the decomposition of harmonic measure in \cite{harmonic_measure_1987}. 

Before the asymptotic sharpness results can be shown, we first introduce the discrete Green function used in \cite{harmonic_measure_1987} and quote some of its basic properties.  

\begin{lemma}(Lemma 1 of  \cite{harmonic_measure_1987})
The series 
\beq
\label{discrete Green}
a(x)=\sum_{n=0}^\infty [P_0(S_n=0)-P_n(S_n=x)]
\eeq
converge for each $x\in \ZZ^2$, and the function $a(\cdot)$ has the following properties:

\beq
\label{property_a_1}
a(x)\ge 0, \ \forall x\in \ZZ^2, \ a(0)=0,
\eeq 

\beq
\label{property_a_2}
a\big((\pm 1,0)\big)=a\big((0,\pm 1)\big)=1
\eeq 

\beq
\label{property_a_3}
E_x[a(S_1)]-a(x)=\delta(x,0),
\eeq 
so $a(S_{n\wedge \tau_v}-v)$ is a nonnegative martingale, where $\tau_v=\tau_{\{v\}}$, for any $v\in \ZZ^2$. And there is some suitable $C_0,C_1<\infty$ such that for all $x\not=0$, 
\beq
\label{property_a_5}
x^2\left|a(x)-\frac{1}{2\pi} \log|x|-C_0 \right|\le C_1.
\eeq
\end{lemma}


Here for any positive $R$ we use the notation 
$$
\tau_R=\tau_{\partial^{out}B(0,R)}.
$$
We have

\begin{lemma}(Inverse of Lemma 3 of  \cite{harmonic_measure_1987})
\label{lemma: inverse_3}
Let $D\subset \{u: \ |u|\le r\}$ contain the origin. Then for all $R$ sufficiently larger than $r$ one has uniformly in $v\in \partial^{out}B(0,R)$ and in $D$
\beq
P_v\left(\tau_D<\tau_R\right)\le C [R \log (R)]^{-1}. 
\eeq
\end{lemma}

\begin{proof}
For any $v\in \partial^{out}B(0,R)$, there must be at least one point among its 4 neighbors within $B(0,R)$. And for each such point $w$, we use the same martingale 
$$
Y_n=a\left(S_{n\wedge \tau_0} \right)
$$
as the Lemma 3 of  \cite{harmonic_measure_1987}, with $S_0=w$. Then since $0\in D$, we have that the stopping time $\sigma=\tau_D\wedge \tau_R\le \tau_0$, and  
$$
\begin{aligned}
a(w)&=E_w[a(S_\sigma)\mathbb{1}_{\tau_D< \tau_R}]+E_w[a(S_\sigma)\mathbb{1}_{\tau_R< \tau_D}]\\
&=P_w(\tau_D< \tau_R) E_w[a(S_\sigma)|\tau_D< \tau_R]+P_w(\tau_R< \tau_D) E_w[a(S_\sigma)|\tau_R< \tau_D]\\
&=E_w[a(S_\sigma)|\tau_R< \tau_D]-P_w(\tau_D< \tau_R)\left(E_w[a(S_\sigma)|\tau_R< \tau_D]- E_w[a(S_\sigma)|\tau_D< \tau_R]\right).
\end{aligned}
$$
Thus we have
\beq
\label{inverse_3_1}
P_w(\tau_D< \tau_R)=\frac{E_w[a(S_\sigma)|\tau_R< \tau_D]-a(w)}{E_w[a(S_\sigma)|\tau_R< \tau_D]- E_w[a(S_\sigma)|\tau_D< \tau_R]}.
\eeq
Note that under $\{\tau_R< \tau_D\}$, $S_\sigma\in \partial^{out}B(0,R)$, which implies $R\le |S_\sigma|\le R+1$. We have by \eqref{property_a_5} for sufficiently large $R$
\beq
\label{inverse_3_2}
\begin{aligned}
E_w[a(S_\sigma)|\tau_R< \tau_D]&\le E\left[\frac{1}{2\pi}\log(|S_\sigma|)+C_0+\frac{C_1}{|S_\sigma|^2}\Big|\tau_R< \tau_D \right]\\
&\le \frac{1}{2\pi}\log(R+1)+C_0+\frac{C_1}{R^2}\\
&\le \frac{1}{2\pi}\log(R)+\frac{1}{2\pi R}+C_0+\frac{C_1}{R^2}.
\end{aligned}
\eeq
The last inequality is from the fact that $\log(1+x)\le x$. On the other hand, we can also have
\beq
\label{inverse_3_3}
\begin{aligned}
E_w[a(S_\sigma)|\tau_R< \tau_D]&\ge E\left[\frac{1}{2\pi}\log(|S_\sigma|)+C_0-\frac{C_1}{|S_\sigma|^2}\Big|\tau_R< \tau_D \right]\\
&\ge \frac{1}{2\pi}\log(R)+C_0-\frac{C_1}{R^2}. 
\end{aligned}
\eeq
Similarly, note that $R-1\le |w|\le R$, we have sufficiently large $R$
\beq
\label{inverse_3_4}
\begin{aligned}
a(w)&\ge \frac{1}{2\pi}\log(|a_w|)+C_0-\frac{C_1}{|a_w|^2}\\
&\ge \frac{1}{2\pi}\log(R-1)+C_0-\frac{2C_1}{R^2}\\
&\ge \frac{1}{2\pi}\log(R)-\frac{1}{\pi R}+C_0-\frac{2C_1}{R^2}.
\end{aligned}
\eeq
Now the last inequality is from the fact that $\log(1-R^{-1})=-R^{-1}+O(-R^{-2})>-2R^{-1}$ for sufficiently large $R$. Finally, we can use the same argument and have for sufficiently large $R$
\beq
\label{inverse_3_5}
\begin{aligned}
E_w[a(S_\sigma)|\tau_D< \tau_R]&= E_w[a(S_\sigma)\mathbb{1}_{S_\sigma\not=0}|\tau_D< \tau_R]\\
&\le \frac{1}{2\pi}\log(r)+C_0+C_1\\
&\le \frac{1}{4\pi}\log(R)+C_0-\frac{C_1}{R^2}.
\end{aligned}
\eeq
From \eqref{inverse_3_1} one immediately has that 
\beq
\label{inverse_3_6}
P_w(\tau_D< \tau_R)\le \frac{\overline E_w[a(S_\sigma)|\tau_R< \tau_D]-\underline a(w)}{\underline E_w[a(S_\sigma)|\tau_R< \tau_D]- \overline E_w[a(S_\sigma)|\tau_D< \tau_R]},
\eeq
where $\overline{(\cdot)}$ stands for an upper bound while $\underline{(\cdot)}$ for a lower bound. Then substitute \eqref{inverse_3_2}-\eqref{inverse_3_5} into \eqref{inverse_3_6}, we have
\beq
\label{inverse_3_7}
P_w(\tau_D< \tau_R)\le \frac{\frac{3}{2\pi R}+\frac{3C_1}{R^2}}{\frac{1}{4\pi}\log(R)}\le 7 [R \log R]^{-1}
\eeq
for sufficiently large $R$. Finally note that 
$$
P_v\left(\tau_D<\tau_R\right)\le \sup_{w\in B(0,R): \ |w-v|=1} P_w(\tau_D< \tau_R).
$$
Thus the proof of this lemma is complete. 
\end{proof}

Our next lemma gives an ``inverse" for Lemma 4 of of  \cite{harmonic_measure_1987}, under a slightly stronger condition. 

\begin{lemma}(Inverse of Lemma 4 of  \cite{harmonic_measure_1987})
\label{lemma: inverse_4}
There are constants $3<C_2<\infty$ and $c_2>0$ such that for all $r$ and $R$ sufficiently larger than $r$, any $D\subset \{u: \ |u|\le r\}$, and for any $z\in \partial^{out} B(0, C_2\cdot r)$, we have
\beq
P_z(\tau_R<\tau_D)\ge \frac{c_2} {\log(R)}.
\eeq
\end{lemma}

\begin{proof}
Again, we first consider the same martingale as in the original lemma. For each $z\in \partial^{out} B(0, C_2\cdot r)$ define
$$
Z_n=\sum_{v\in D}a(S_{n\wedge \tau_D}-v)
$$
with $S_0=z$. Recall that $\sigma=\tau_D\wedge \tau_R\le \tau_D$. Using the same argument as in Lemma 4 of  \cite{harmonic_measure_1987} and Lemma \ref{lemma: inverse_3} above, we have 
$$
\begin{aligned}
\sum_{v\in D}a(z-v)&=P_z(\tau_R<\tau_D) \left( \sum_{v\in D}E[a(S_{\sigma}-v)| \tau_R<\tau_D]\right)\\
&+[1-P_z(\tau_R<\tau_D)] \left( \sum_{v\in D}E[a(S_{\sigma}-v)| \tau_D<\tau_R]\right),
\end{aligned}
$$
which gives us
\beq
\label{inverse_4_1}
\begin{aligned}
P_z(\tau_R<\tau_D)&=\frac{\sum_{v\in D}a(z-v)-\sum_{v\in D}E[a(S_{\sigma}-v)| \tau_D<\tau_R]}{\sum_{v\in D}E[a(S_{\sigma}-v)| \tau_R<\tau_D]-\sum_{v\in D}E[a(S_{\sigma}-v)| \tau_R<\tau_D]}\\ 
\\
&\ge \frac{\sum_{v\in D}a(z-v)-\sum_{v\in D}E[a(S_{\sigma}-v)| \tau_D<\tau_R]}{\sum_{v\in D}E[a(S_{\sigma}-v)| \tau_R<\tau_D]}\\
\\
&\ge \frac{\underline\sum_{v\in D}a(z-v)-\overline\sum_{v\in D}E[a(S_{\sigma}-v)| \tau_D<\tau_R]}{\overline\sum_{v\in D}E[a(S_{\sigma}-v)| \tau_R<\tau_D]}\\
\end{aligned}
\eeq
where $\overline{(\cdot)}$ again stands for an upper bound while $\underline{(\cdot)}$ for a lower bound. Then for any $z\in \partial^{out} B(0, C_2\cdot r)$ and any $v\in D\subset B(0,r)$, $|z-v|\ge (C_2-1)r$, which implies 
$$
\begin{aligned}
a(z-v)&\ge \frac{1}{2\pi}\log(|z-v|)+C_0-\frac{C_1}{|z-v|^2}\\
&\ge \frac{1}{2\pi}\log(r)+\frac{\log(C_2-1)}{2\pi}+C_0-C_1.
\end{aligned}
$$
Taking summation over all $v\in D$, we have
\beq
\label{inverse_4_2}
\sum_{v\in D}a(z-v)\ge \frac{|D|}{2\pi}\log(r)+\left[\frac{\log(C_2-1)}{2\pi}+C_0-C_1 \right]\cdot|D|. 
\eeq
Then under $\{\tau_D<\tau_R\}$, $S_{\sigma}\in D$, which implies that $|S_{\sigma}-v|\le 2r$ for all $v\in D\subset B(0,r)$. Then we have 
$$
\begin{aligned}
E[a(S_{\sigma}-v)| \tau_D<\tau_R]&=E[a(S_{\sigma}-v)\mathbb{1}_{S_{\sigma}\not=v} | \tau_D<\tau_R]\\
&\le E\left[\left(\frac{1}{2\pi}\log(|S_{\sigma}-v|)+C_0+\frac{C_1}{|S_{\sigma}-v|^2}\right)\mathbb{1}_{S_{\sigma}\not=v}\Big | \tau_D<\tau_R \right]\\
&\le \frac{1}{2\pi}\log(r)+\left[\frac{\log(2)}{2\pi}+C_0+C_1\right]. 
\end{aligned}
$$
Again taking summation over all $v\in D$, we have
\beq
\label{inverse_4_3}
\sum_{v\in D}E[a(S_{\sigma}-v)| \tau_D<\tau_R]\le \frac{|D|}{2\pi}\log(r)+\left[\frac{\log(2)}{2\pi}+C_0+C_1\right]\cdot|D|.
\eeq
Finally, under $\{\tau_R<\tau_D\}$, $S_{\sigma}\in \partial^{out}B(0,R)$, so for any $v\in D$, $R-r\le |S_{\sigma}-v|\le R+r$. Thus 
$$
E[a(S_{\sigma}-v)| \tau_R<\tau_D]\le \frac{1}{2\pi}\log(R+r)+C_0+\frac{C_1}{(R-r)^2}\le \frac{1}{\pi}\log(R)
$$
for all $R$ sufficiently larger than $r$. Thus we have 
\beq
\label{inverse_4_4}
\sum_{v\in D}E[a(S_{\sigma}-v)| \tau_R<\tau_D]\le \frac{|D|}{\pi}\log(R).
\eeq
Now, we can substitute \eqref{inverse_4_2}-\eqref{inverse_4_4} into \eqref{inverse_4_1},
\beq
\begin{aligned}
P_z(\tau_R<\tau_D)&\ge \frac{\left[\frac{\log(C_2-1)}{2\pi}-\frac{\log(2)}{2\pi}-2C_1 \right]\cdot|D|}{\frac{|D|}{\pi}\log(R)}=\frac{\log(\frac{C_2-1}{2})-4\pi C_1}{2\log(R)},
\end{aligned}
\eeq
and let $C_2=2\exp(4\pi C_1+1)+1$, and $c_2=1/2$ to complete the proof. 

\end{proof}

With Lemma \ref{lemma: inverse_3} and Lemma \ref{lemma: inverse_4}, we are now able to show the asymptotic sharpness of (2.15) in \cite{harmonic_measure_1987} and have 
\begin{lemma}
\label{lemma: 2.15}
There is a constant $c_3>0$ such that for all $D\subset \{u: \ |u|\le r\}$ contains the origin, and any $y\in D$
$$
\mu_D(y)\ge c_3 P_y(\tau_{C_2\cdot r}<\tau_D).
$$
Here $\mu_D(\cdot)$ is the harmonic measure on $\ZZ^2$ associated with $D$, 
$$
\mu_D(y)=\lim_{|z|\to\infty} P_z(S_{\tau_D}=y). 
$$ 

\end{lemma}
\begin{proof}
Here we follow exactly the same decomposition according to first hitting position, as in \cite{harmonic_measure_1987}. There is a $c>0$ such that for all sufficiently large $R$, 
$$
\mu_D(y)\ge \frac{c}{R}\sum_{z\in \partial^{out} B(0, R)}P_z(S_{\tau_D}=y). 
$$
And for any $z\in \partial^{out} B(0, R)$, we again have 
\begin{align*}
P_z(S_{\tau_D}=y)&=\sum_{n=1}^\infty P_z(\tau_D=n, S_{n}=y)\\
&=\sum_{n=1}^\infty P_y(\tau_D>n, S_{n}=z)\\
&=E_y\big[\text{\# of visits to $z$ in }[0,\tau_D]\big]. 
\end{align*}
Thus we have
\beq
\label{inverse_5_0}
\mu_D(y)\ge \frac{c}{R} E_y\big[\text{\# of visits to $\partial^{out} B(0, R)$ in }[0,\tau_D]\big].
\eeq
Apply strong Markov property on the expectation, we have 
\beq
\label{inverse_5_1}
\begin{aligned}
E_y&\big[\text{\# of visits to $\partial^{out} B(0, R)$ in }[0,\tau_D]\big]\\
&=\sum_{w\in  \partial^{out} B(0, R)}P_y(\tau_R<\tau_D, S_{\tau_R}=w) E_w \big[\text{\# of visits to $\partial^{out} B(0, R)$ in }[0,\tau_D]\big]. 
\end{aligned}
\eeq
Then for each $w\in  \partial^{out} B(0, R)$, 
\beq
\label{inverse_5_2}
\begin{aligned}
E_w &\big[\text{\# of visits to $\partial^{out} B(0, R)$ in }[0,\tau_D]\big]\\
&=\sum_{k=0}^\infty P_w(\text{$S$ returns at least $k$ times to $\partial^{out} B(0, R)$ before $\tau_D$})\\
&\ge \sum_{k=0}^\infty \left(\inf_{v\in \partial^{out} B(0, R)} \left\{P_v(\tau_R<\tau_D)\right\} \right)^k\\
&=\frac{1}{1-\inf_{v\in \partial^{out} B(0, R)} \left\{P_v(\tau_R<\tau_D)\right\} }\\
&=\frac{1}{\sup_{v\in \partial^{out} B(0, R)} \left\{P_v(\tau_D<\tau_R)\right\} }.
\end{aligned}
\eeq
Note that in Lemma \ref{lemma: inverse_3} we proved that 
$$
\sup_{v\in \partial^{out} B(0, R)} \left\{P_v(\tau_D<\tau_R)\right\}\le C [R \log (R)]^{-1}.
$$
Plug it to \eqref{inverse_5_2}, we have for all $w\in  \partial^{out} B(0, R)$, 
\beq
\label{inverse_5_3}
\begin{aligned}
E_w &\big[\text{\# of visits to $\partial^{out} B(0, R)$ in }[0,\tau_D]\big]\ge \frac{1}{C} R\log(R). 
\end{aligned}
\eeq
Then combining \eqref{inverse_5_0}, \eqref{inverse_5_1} and \eqref{inverse_5_3}, 
\beq
\label{inverse_5_4}
\begin{aligned}
\mu_D(y)&\ge \frac{c}{R}\cdot \frac{1}{C} R\log(R) \sum_{w\in  \partial^{out} B(0, R)}P_y(\tau_R<\tau_D, S_{\tau_R}=w)\\
&= \frac{c}{C}\log(R) P_y(\tau_R<\tau_D). 
\end{aligned}
\eeq
Then for $P_y(\tau_R<\tau_D)$, note that for sufficiently large $R>C_2\cdot r$, if a random walk wants to exit $B(0,R)$, it has to exit $B(0,C_2\cdot r)$ first. Thus, again by strong Markov property, we have
$$
P_y(\tau_R<\tau_D)=\sum_{z\in \partial^{out} B(0, C_2\cdot r)} P_y(\tau_{C_2\cdot r}<\tau_D, S_{\tau_{C_2\cdot r}}=z)P_z(\tau_R<\tau_D). 
$$
Since in Lemma \ref{lemma: inverse_4}, we prove that for any $z\in \partial^{out} B(0, C_2\cdot r)$, 
$$
P_z(\tau_R<\tau_D)\ge \frac{c_2} {\log(R)},
$$
we have
\beq
\begin{aligned}
\label{inverse_5_5}
P_y(\tau_R<\tau_D)&\ge\frac{c_2} {\log(R)} \sum_{z\in \partial^{out} B(0, C_2\cdot r)} P_y(\tau_{C_2\cdot r}<\tau_D, S_{\tau_{C_2\cdot r}}=z)\\
&=\frac{c_2} {\log(R)}P_y(\tau_{C_2\cdot r}<\tau_D).
\end{aligned}
\eeq
Combining \eqref{inverse_5_4} and \eqref{inverse_5_5}, and let $c_3=c\cdot c_2/C$, the proof of Lemma \ref{lemma: 2.15} is complete. 
\end{proof}

\subsection{Proof of Theorem \ref{theorem: uniform_path}}

Now we have the tools we need to finish the proof of Theorem \ref{theorem: uniform_path}. Recall that $B_n=L_0\cup \mathcal{Q}_n$, and that by \eqref{measure to expectation}, strong Markov property, and \eqref{1_dimensional_mean}
\beq
\label{Theorem_1_1}
\begin{aligned}
&H_{B_n,N}(y_n)\\
&=E_{y_n}\Big[\text{number of visits to $L_N$ in time interval } [0,\tau_{B_n}]\Big]\\
&=\sum_{w\in L_N}P_{y_n}(\tau_{L_N}< \tau_{B_n}, S_{\tau_{L_N}}=w) E_w\Big[\text{number of visits to $L_N$ in time interval } [0,\tau_{B_n}]\Big]\\
&\le\sum_{w\in L_N}P_{y_n}(\tau_{L_N}< \tau_{B_n}, S_{\tau_{L_N}}=w) E_w\Big[\text{number of visits to $L_N$ in time interval } [0,\tau_{L_0}]\Big]\\
&=4N\cdot P_{y_n}(\tau_{L_N}< \tau_{B_n}). 
\end{aligned}
\eeq
So in order to show \eqref{truncated path} and thus Theorem \ref{theorem: uniform_path}, it is sufficient to prove that 
\beq
\label{Theorem_1_2}
P_{y_n}(\tau_{L_N}< \tau_{B_n})\le \frac{C n^{1/2}}{N}. 
\eeq
To show \eqref{Theorem_1_2}, define $r_n=2n$, $\mathcal{S}_n=\partial^{out}B(y_n,C_2\cdot r_n)\cap \{(x,y)\in \ZZ^2, y\ge 1\}$. Note that if a simple random walk starting at $y_n$ wants to reach $L_N$ before returning to $B_n$, it has to visit some point in $\mathcal{S}_n$ first. Thus once again by strong Markov property,
\beq
\label{Theorem_1_3}
P_{y_n}(\tau_{L_N}< \tau_{B_n})=\sum_{z\in \mathcal{S}_n}P_{y_n}(\tau_{\mathcal{S}_n}< \tau_{B_n}, S_{\tau_{\mathcal{S}_n}}=z)P_z(\bar\tau_{L_N}< \bar\tau_{B_n}). 
\eeq
Note that for each $z\in \mathcal{S}_n$, by \eqref{basic 2} and the fact that $L_0\subset B_n$, 
$$
P_z(\bar\tau_{L_N}< \bar\tau_{B_n})\le P_z(\bar\tau_{L_N}<\bar\tau_{L_0})\le \frac{(2C_2+1)n}{N}. 
$$
Plugging this uniform upper bound into \eqref{Theorem_1_3}, we now have
$$
P_{y_n}(\tau_{L_N}< \tau_{B_n})\le P_{y_n}(\tau_{\mathcal{S}_n}< \tau_{B_n}) \cdot  \frac{(2C_2+1)n}{N}. 
$$
Thus for Theorem \ref{theorem: uniform_path} it is sufficient to show that
\beq
\label{Theorem_1_4}
P_{y_n}(\tau_{\mathcal{S}_n}< \tau_{B_n})\le C n^{-1/2}.
\eeq
Noting that $\mathcal{S}_n\subset \partial^{out}B(y_n,C_2\cdot r_n)$, and that $\mathcal{Q}_n\subset B_n$, then by \eqref{basic 1} and \eqref{basic 2},
$$
P_{y_n}(\tau_{\mathcal{S}_n}< \tau_{B_n})\le P_{y_n}(\tau_{\partial^{out}B(y_n,C_2\cdot r_n)}< \tau_{\mathcal{Q}_n}).
$$
Since simple random walk is translation invariant, 
$$
P_{y_n}(\tau_{\partial^{out}B(y_n,C_2\cdot r_n)}< \tau_{\mathcal{Q}_n})=P_0(\tau_{C_2\cdot r_n}<\tau_{D_n}),
$$
where $D_n=\mathcal{Q}_n-y_n$, which is a connected subset of $B(0,r_n)$ containing 0. Then apply Lemma \ref{lemma: 2.15} on $y=0\in D_n\subset B(0,r_n)$, we have
$$
P_0(\tau_{C_2\cdot r_n}<\tau_{D_n})\le \frac{1}{c_3} \mu_{D_n}(0).
$$
Finally by Theorem 1 (the only theorem) of \cite{harmonic_measure_1987}, and the fact that $0\in D_n$, with $D_n$ connected and $r(D_n)\in [n,2n]$. 
$$
\mu_{D_n}(0)\le C n^{-1/2},
$$
which finishes the proof of Theorem \ref{theorem: uniform_path}. \qed

\section{Subset maximizing the stationary harmonic measure}
\label{section: maximum}

In this section we prove Theorem \ref{theorem: maximum_path}. Then together with the uniform upper bound we had in Theorem \ref{theorem: uniform_path}, one can see that $U_n=V_n\cup L_0$ is the subset maximizing harmonic measure up to multiplying a constant. 

Before we start with the details, an outline of the proof of Theorem \ref{theorem: maximum_path} is presented. See also Figure \ref{fig:outfig}. The detailed proof will piece together everything we need in the list below, although the order that each lemma is proved may not be precisely consistent with the outline. 
\begin{enumerate}[(i)]
\item We have found that $H_{U_n,N}(y_n)$ equals to the expected number of visits to $L_N$ before a simple random walk $S$ starting from $y_n$ returns to $U_n$. If the random walk reaches $L_N$ first before returning to $U_n$, the expected  number of (re-)visits is $4N+o(N)$. 
\item For $S$ to reach $L_N$ first before returning to $U_n$, it has to reach $L_{2n}$ first. Once it reaches $L_{2n}$, the probability of success from there is at least of oder $n/N$. 
\item If $S$ reached the upper outer boundary of the $L_1$ ball $B_1(y_n,n/3)=\{|x|+|y-n|\le n/3\}$ before returning to $V_n$, by the invariance principle there is a positive probability for it to continue to $L_{2n}$ before returning to $U_n$.
\item The probability that $S_n$ exits $B_1(y_n,n/3)$ before returning to $V_n$ is at least $O(n^{-1/2})$.
\item Given $S_n$ exits $B_1(y_n,n/3)$ before returning to $V_n$, it is more likely to exit from the upper half than the lower half. 
\end{enumerate}

Without loss of generality, we only need to prove this theorem for $n$ sufficiently large and $N$ sufficiently larger than $n$. 

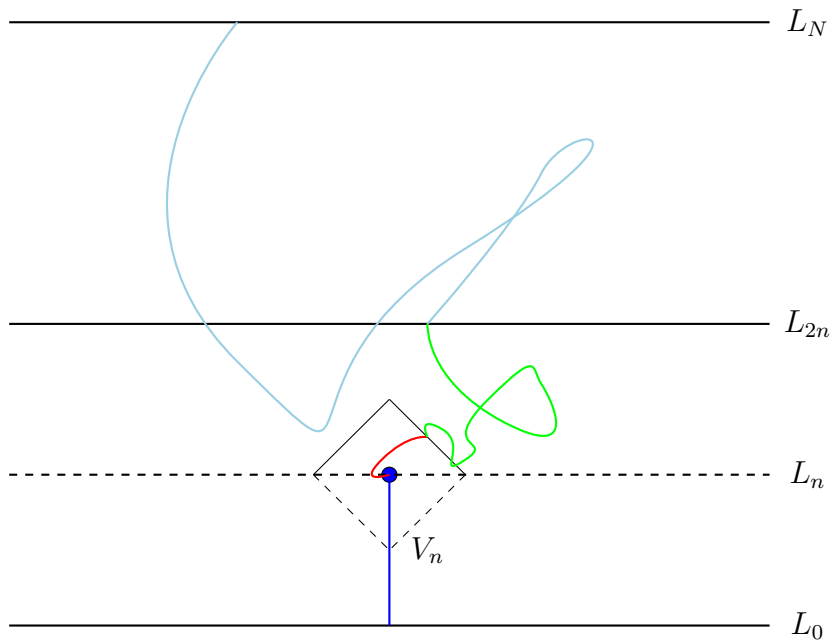
\begin{figure}[h!]
\centering
\input{outline}
\caption{Outline for the lower bound}
\label{fig:outfig}
\end{figure} 

\subsection{Lower bound on escaping probability}

To show that $S_n$ exits $B_1(y_n,n/3)$ before returning to $V_n$ with probability at least $O(n^{-1/2})$, we need to prove one more asymptotic sharpness result which is basically an inverse of Lemma 6 in \cite{harmonic_measure_1987} with $y=0$. Here we first introduce the definition of the infinite range 1-dimensional random walk in their problem and quote its properties: 

Let $S_n, \ n\ge 0$ be a 2-dimensional simple random walk starting from the $y-$ (or equivalently $x-$) axis. Define the stopping times $\sigma_0=0$, 
$$
\sigma_{k+1}=\inf\{n>\sigma_k, \ S_n\in y\text{-axis} \}, 
$$ 
and
$$
T_k=y\text{-coornidate of } S_{\sigma_k},  \ \ Y_{k+1}=T_{k+1}-T_k.
$$
Note that now $T_k$ is a 1-dimensional random walk although each of its steps has no well defined expectation. Moreover, define $\rho$ to be a stopping time for $T$
$$
\rho=\inf\{k\ge 1, T_k\le 0\}.
$$
We can also define the Green function for the random walk $T$ stopped at $\rho$ as
\beq
\label{Green function T 1d}
G(j,l)=E_j[\text{number of visits by $T$ to $l$ before $\rho$}], \ \ j,l>0. 
\eeq
The following properties of $G(\cdot,\cdot)$ have been proved 
\begin{lemma}(Lemma 5 in \cite{harmonic_measure_1987})
\label{lemma_5}
For $j,l>0$,
\beq
G(j,l)=\sum_{n=1}^{j\wedge l} v(j-n)v(l-n)
\eeq
for some numbers $v(\cdot)$ satisfying 
\begin{align}
& v(n)\ge 0,\\
& V(n)=\sum_{i=0}^n v(k)\sim C \sqrt{n},\\
& V(T_{n\wedge \rho})\text{ is a nonnegative martingale under }P_j, \ j>0.
\end{align}
\end{lemma}

Moreover, in Equation (2.27) of  \cite{harmonic_measure_1987} it has been proved that 
\beq
\label{inverse_6_0}
P(Y_1<-n)\sim \frac{1}{2n},
\eeq
and by symmetry
\beq
\label{inverse_6_0.5}
P(Y_1>n)\sim \frac{1}{2n}.
\eeq
Now we have all the tools needed to get the following lemma:
\begin{lemma}(Inverse of Lemma 6 in \cite{harmonic_measure_1987})
\label{lemma: 6}
There is a constant $c>0$ such that for all $r\ge 1$, one has
\beq
\label{inverse_6_1}
P_0(T_\rho<-r)\ge c r^{-1/2}. 
\eeq
\end{lemma}

\begin{proof}
Condition on the location of $T_1$, we have
\beq
\label{inverse_6_2}
\begin{aligned}
P_0(T_\rho<-r)&=P_0(Y_1<-r)+\sum_{j=1}^\infty P(Y_1=j) P_j(T_\rho<-r)\\
&\ge \sum_{j=1}^\infty P(Y_1=j) P_j(T_\rho<-r). 
\end{aligned}
\eeq
For each $j>0$,
\begin{align*}
P_j(T_\rho<-r)=\sum_{l=1}^\infty\sum_{n=1}^\infty P_j(\rho=n, T_{n-1}=l, T_n<-r), 
\end{align*}
while for each $n,j,l\ge 0$, 
$$
\begin{aligned}
P_j(\rho=n, T_{n-1}=l, T_n<-r)&=P_j(\rho>n-1, T_{n-1}=l,T_n<-r)\\
&=P_j(\rho>n-1, T_{n-1}=l)P_l(T_1<-r)\\
&=P_j(\rho>n-1, T_{n-1}=l)P(Y_1<-r-l).
\end{aligned}
$$
Taking the summation we have 
\beq
\label{inverse_6_3}
P_j(T_\rho<-r)=\sum_{l=1}^\infty G(j,l)P(Y_1<-r-l)
\eeq
and thus by \eqref{inverse_6_1}-\eqref{inverse_6_3}
\beq
\label{inverse_6_4}
\begin{aligned}
P_0(T_\rho<-r)&\ge \sum_{j=1}^\infty P(Y_1=j) \sum_{l=1}^\infty G(j,l)P(Y_1<-r-l)\\
&\ge c\sum_{j=1}^\infty P(Y_1=j) \sum_{l=1}^\infty G(j,l)\frac{1}{r+l}. 
\end{aligned}
\eeq
Then by Lemma \ref{lemma_5} we have 
\beq
\label{inverse_6_5}
\begin{aligned}
P_0(T_\rho<-r)&\ge c\sum_{j=1}^\infty P(Y_1=j) \sum_{l=1}^\infty G(j,l)\frac{1}{r+l}\\
&\ge  c\sum_{j=1}^\infty P(Y_1=j) \sum_{l=1}^\infty \frac{1}{r+l}\sum_{n=1}^{j\wedge l} v(j-n)v(l-n)\\
&= c\sum_{j=1}^\infty P(Y_1=j) \sum_{n=1}^j v(j-n)\sum_{l=n}^\infty v(l-n) \frac{1}{r+l}.
\end{aligned}
\eeq
Noting that for each $l\ge n$, 
$$
 \frac{1}{r+l}=\frac{1}{(r+l)(r+l+1)}+\frac{1}{(r+l+1)(r+l+2)}+\cdots,
$$
with summation by parts and Lemma \ref{lemma_5} we have for each $n$
\beq
\label{inverse_6_6}
\begin{aligned}
\sum_{l=n}^\infty v(l-n) \frac{1}{r+l}&=\sum_{m=0}^\infty \frac{V(m)}{(r+n+m)(r+n+m+1)}\\
&\ge c\sum_{m=0}^\infty \frac{\sqrt{m}}{(r+n+m)(r+n+m+1)}\\
&= c\sum_{m=1}^\infty \frac{\sqrt{m}-\sqrt{m-1}}{r+n+m}\\
&\ge \frac{1}{2} \sum_{m=1}^\infty \frac{1}{(r+n+m)^{3/2}}\\
&\ge \frac{1}{2} \int_{r+n+1}^\infty \frac{1}{x^{3/2}}\ge \frac{1}{2}(r+n)^{-1/2}.
\end{aligned}
\eeq
Then noting that for $l=j=1$,
$$
G(j,l)=\sum_{n=1}^{j\wedge l} v(j-n)v(l-n)=v(0)^2,
$$
and that by definition $G(1,1)\ge1>0$, we have $v(0)\ge1$. Moreover, note that $P(Y_1=1)\ge P_0(S_1=(0,1))=1/4$. Since all the terms in \eqref{inverse_6_5} are nonnegative, let $j=n=1$, we have 
\beq
\label{inverse_6_7}
\begin{aligned}
P_0(T_\rho<-r)&\ge cP(Y_1=1)v(0)\sum_{l=1}^\infty v(l-1) \frac{1}{r+l}\\
&\ge \frac{c}{8}(r+1)^{-1/2}\ge \frac{c}{16}r^{-1/2}. 
\end{aligned}
\eeq
Thus the proof of Lemma \ref{lemma: 6} is complete. 

\end{proof}

With Lemma \ref{lemma: 6} the proof for the desired lower bound of escaping probability is straightforward. Recall that we have a 2-dimensional simple random walk starting at $y_n=(0,n)$. Define $V'_n=\{(0,y), \  n-[n/3]\le y\le n\}$ and
$$
\mathcal{S}_{1,n}=\partial B_1(y_n,[n/3]). 
$$
Here note that for $L_1$ ball $B_1(y_n,[n/3])$ we do not need to specify if the boundary is in or out. Then for $C'_n=\{(0,y), \  y<n-[n/3]\}$, note that for a 2-dimensional simple random walk starting at $y_n=(0,n)$ we always have
$$
\tau_{\mathcal{S}_{1,n}}<\tau_{C'_n}.
$$
Thus for the escaping probability we want to bound from below, we have
$$
P_{y_n}\left(\tau_{\mathcal{S}_{1,n}}<\tau_{U_n}\right)=P_{y_n}\left(\tau_{\mathcal{S}_{1,n}}<\tau_{V'_n}\right)\ge P_{y_n}\left(\tau_{C'_n}<\tau_{V'_n}\right).
$$
By the translation invariance of $S$, 
$$
P_{y_n}\left(\tau_{C'_n}<\tau_{V'_n}\right)=P_{0}\left(\tau_{C'_n-y_n}<\tau_{V'_n-y_n}\right).
$$
Note that $C'_n-y_n=\{(0,y), \  y<-[n/3]\}$ and that $V'_n-y_n=\{(0,y), \  -[n/3]\le y\le 0\}$. Let $r=[n/3]$ one has
$$ 
P_{0}\left(\tau_{C'_n-y_n}<\tau_{V'_n-y_n}\right)= P_0(\rho<-r)\ge c\sqrt{3} n^{-1/2}. 
$$
Thus we have 
\beq
\label{escaping probability}
P_{y_n}\left(\tau_{\mathcal{S}_{1,n}}<\tau_{U_n}\right)=P_{y_n}\left(\tau_{\mathcal{S}_{1,n}}<\tau_{V'_n}\right) \ge c\sqrt{3} n^{-1/2}. 
\eeq

And similarly, if we look back at the harmonic measure on $\ZZ^2$ and let $r'=(C_2+1)r$. Then since 
$$
\tau_{C_2\cdot r}<\tau_{\{(-\infty, r')\times 0\}}
$$
Lemma \ref{lemma: 6} gives us
$$
P_0(\tau_{C_2\cdot r}<\tau_C)\ge P_0(T_\rho<-r')\ge c r^{-1/2}. 
$$
Then by Lemma \ref{lemma: 2.15},
$$
\mu_C(0)\ge c_3 P_0(\tau_{C_2\cdot r}<\tau_C)\ge cr^{-1/2}. 
$$
Thus we have also proved 
\begin{corollary}
\label{corllary_1987}
Among all connected subset $B\subset \ZZ^2$ containing the origin, with $r(B)=r$ and all $y\in B$, $\mu_B(y)$ is maximized (up to multiplying a constant) by when $B=[-r,r]\times 0$ and $y=(r,0)$.
\end{corollary}

\begin{remark}
Note that Corollary \ref{corllary_1987}, complements Kesten's paper \cite{harmonic_measure_1987}, by showing that the straight line is a maximizer up to a multiplicative constant in $\mathbb{Z}^2$.
\end{remark}

\subsection{Spatial distribution at the escaping time}

Now with Lemma \ref{lemma: 6} shows that a 2-dimensional simple random walk starting at $y_n$ will escape $B_1(y_n,[n/3])$ before returning to $V_n'$ and thus $U_n$ with probability at least some constant times $n^{-1/2}$. We next show that, given the random walk successfully escapes, it is more likely to escape from the upper half of $\mathcal{S}_{1,n}$ that the lower half of it. To make it precise, define
$$
\mathcal{S}^U_{1,n}=\mathcal{S}_{1,n}\cap\{(x,y), y\ge n\},
$$
and
$$
\mathcal{S}^L_{1,n}=\mathcal{S}_{1,n}\cap\{(x,y), y\le n\}.
$$
Then for stopping time $\sigma=\tau_{\mathcal{S}_{1,n}}\wedge \tau_{V'_n}$, we want to show 
\beq
\label{Spatial_1}
P_{y_n}\left(\tau_{\mathcal{S}_{1,n}}< \tau_{V'_n}, S_{\sigma}\in \mathcal{S}^U_{1,n}\right)\ge P_{y_n}\left(\tau_{\mathcal{S}_{1,n}}< \tau_{V'_n}, S_{\sigma}\in \mathcal{S}^L_{1,n}\right).
\eeq 
To show this we can again use translation invariance to move everything centered at 0. For integer $m\ge 1$, let
$$
A_m^+=\{(x,y)\in \ZZ^2, \ x+y=m, \ x\in [0,m] \}\cup \{(x,y)\in \ZZ^2, \ -x+y=m, \ x\in [-m,0] \}
$$
and 
$$
A_m^-=\{(x,y)\in \ZZ^2, \ x+y=-m, \ x\in [-m,0,] \}\cup \{(x,y)\in \ZZ^2, \ -x+y=-m, \ x\in [0,m] \}
$$
be the upper and lower half of $\partial B_1(0,m)$. Then define $C^-_m=\{(0,-i), i=0,1,\cdots, m\}$, and $C^+_m=\{(0,i), i=0,1,\cdots, m\}$. To show \eqref{Spatial_1}, it suffices to prove the following lemma:

\begin{lemma}
\label{lemma: spatial}
For all integer $m$, define set
$$
E_m^-=A_m^+\cup A_m^-\cup C_m^-
$$
and stopping time 
$$
\sigma^-_m=\tau_{E_m^-}=\tau_{A_m^+}\wedge\tau_{A_m^-}\wedge \tau_{C_m^-}.
$$
We have
\beq
\label{Spatial_2}
P_0\left(\tau_{A_m^+}=\sigma_m^-\right)\ge P_0\left(\tau_{A_m^-}=\sigma_m^-\right).
\eeq
\end{lemma}
\begin{proof}
For 
$$
E_m=A_m^+\cup A_m^-\cup C_m^-\cup C_m^+,
$$
and stopping time 
$$
\sigma_m=\tau_{E_m}=\tau_{A_m^+}\wedge\tau_{A_m^-}\wedge \tau_{C_m^-}\wedge \tau_{C_m^+},
$$
by symmetry we have 
$$
P_0\left(\tau_{A_m^+}=\sigma_m\right)= P_0\left(\tau_{A_m^-}=\sigma_m\right).
$$
At the same time,
\beq
\label{Spatial_3}
\begin{aligned}
P_0\left(\tau_{A_m^+}=\sigma_m^-\right)&=P_0\left(\tau_{A_m^+}\le\tau_{C_m^+},\tau_{A_m^+}=\sigma_m^-\right)+P_0\left(\tau_{A_m^+}>\tau_{C_m^+},\tau_{A_m^+}=\sigma_m^-\right)\\
&=P_0\left(\tau_{A_m^+}=\sigma_m\right)+P_0\left(\tau_{C_m^+}<\sigma_m^-,\tau_{A_m^+}=\sigma_m^-\right),
\end{aligned}
\eeq
and
\beq
\label{Spatial_4}
\begin{aligned}
P_0\left(\tau_{A_m^-}=\sigma_m^-\right)&=P_0\left(\tau_{A_m^-}\le\tau_{C_m^+},\tau_{A_m^-}=\sigma_m^-\right)+P_0\left(\tau_{A_m^-}>\tau_{C_m^+},\tau_{A_m^-}=\sigma_m^-\right)\\
&=P_0\left(\tau_{A_m^-}=\sigma_m\right)+P_0\left(\tau_{C_m^+}<\sigma_m^-,\tau_{A_m^-}=\sigma_m^-\right).
\end{aligned}
\eeq
Thus it is sufficient to show 
\beq
\label{Spatial_4}
P_0\left(\tau_{C_m^+}<\sigma_m^-,\tau_{A_m^+}=\sigma_m^-\right)\ge P_0\left(\tau_{C_m^+}<\sigma_m^-,\tau_{A_m^-}=\sigma_m^-\right). 
\eeq
Under event $\{\tau_{C_m^+}<\sigma_m^-\}$, let random variable $N_m^+$ be the {\bf last time} $S$ visits $C_m^+$ in $[0,\sigma_m^--1]$. Note that $N_m^+$ is not a stopping time so we cannot use strong Markov property. But we can nonetheless have the decomposition: 
\beq
\label{Spatial_5}
\begin{aligned}
&P_0\left(\tau_{C_m^+}<\sigma_m^-,\tau_{A_m^+}=\sigma_m^-\right)\\
&=\sum_{k=1}^\infty \sum_{\tiny \begin{aligned}&x_1,x_2,\cdots, x_{k-1}\notin E_m^- \\ &x_k\in x_k\in C_m^+\setminus\{0,(0,m)\} \end{aligned}} 
P_0\left(S_1=x_1,\cdots, S_k=x_k, N_m^+=k,\tau_{C_m^+}<\sigma_m^-,\tau_{A_m^+}=\sigma_m^-\right).
\end{aligned}
\eeq
and 
\beq
\begin{aligned}
&P_0\left(\tau_{C_m^+}<\sigma_m^-,\tau_{A_m^-}=\sigma_m^-\right)\\
&=\sum_{k=1}^\infty \sum_{\tiny \begin{aligned}&x_1,x_2,\cdots, x_{k-1}\notin E_m^- \\ &x_k\in x_k\in C_m^+\setminus\{0,(0,m)\} \end{aligned}} 
P_0\left(S_1=x_1,\cdots, S_k=x_k,N_m^+=k,\tau_{C_m^+}<\sigma_m^-,\tau_{A_m^-}=\sigma_m^-\right).
\end{aligned}
\eeq
Note that for each $k$, $x_1,x_2,\cdots, x_{k-1}\notin E_m^-$, and $x_k\in x_k\in C_m^+\setminus\{0,(0,m)\}$, we have
$$ 
\begin{aligned}
&\{S_1=x_1,\cdots, S_k=x_k, N_m^+=k,\tau_{C_m^+}<\sigma_m^-,\tau_{A_m^+}=\sigma_m^-\}\\
=&\{S_1=x_1,\cdots, S_k=x_k, \ S_{k+1+\cdot} \text{ visit $A_m^+$ no later than it first visits $A_m^-\cup C_m^+\cup C_m^-$} \}.
\end{aligned}
$$
So by Markov property, we have
\beq
\label{Spatial_6}
\begin{aligned}
&P_0\left(S_1=x_1,\cdots, S_k=x_k, N_m^+=k,\tau_{C_m^+}<\sigma_m^-,\tau_{A_m^+}=\sigma_m^-\right)\\
&=P_0\left(S_1=x_1,\cdots, S_k=x_k\right) P_{x_k}\left(\tau_{A_m^+}=\sigma_m\right).
\end{aligned}
\eeq
Plugging back in \eqref{Spatial_5} we have 
\beq
\label{Spatial_7}
\begin{aligned}
&P_0\left(\tau_{C_m^+}<\sigma_m^-,\tau_{A_m^+}=\sigma_m^-\right)\\
&=\sum_{k=1}^\infty \sum_{\tiny \begin{aligned}&x_1,x_2,\cdots, x_{k-1}\notin E_m^- \\ &x_k\in x_k\in C_m^+\setminus\{0,(0,m)\} \end{aligned}} 
P_0\left(S_1=x_1,\cdots, S_k=x_k\right) P_{x_k}\left(\tau_{A_m^+}=\sigma_m\right),
\end{aligned}
\eeq
while the same argument for $A_m^-$ gives us 
\beq
\label{Spatial_8}
\begin{aligned}
&P_0\left(\tau_{C_m^+}<\sigma_m^-,\tau_{A_m^-}=\sigma_m^-\right)\\
&=\sum_{k=1}^\infty \sum_{\tiny \begin{aligned}&x_1,x_2,\cdots, x_{k-1}\notin E_m^- \\ &x_k\in x_k\in C_m^+\setminus\{0,(0,m)\} \end{aligned}} 
P_0\left(S_1=x_1,\cdots, S_k=x_k\right) P_{x_k}\left(\tau_{A_m^-}=\sigma_m\right).
\end{aligned}
\eeq
Comparing \eqref{Spatial_7} and \eqref{Spatial_8} term by term, one can see it suffices to show that for all $z=(0,j)\in x_k\in C_m^+\setminus\{0,(0,m)\}$, 
\beq
\label{Spatial_9}
P_{z}\left(\tau_{A_m^+}=\sigma_m\right)\ge  P_{z}\left(\tau_{A_m^-}=\sigma_m\right). 
\eeq
To show \eqref{Spatial_9}, one first sees that under $\{\tau_{A_m^+}=\sigma_m\}$ or $\{\tau_{A_m^-}=\sigma_m\}$, a random walk starting at $z$ has to move horizontally at the first step then remain in the right or left half triangle of $B_1(0,m)$ until it exits from $A_m^+$ or $A_m^-$. Then for all integer $i\in [0,m]$ we define
$$
C_{m,i}=\{(0,y),  \ 2i-m\le y\le m\},
$$
\begin{align*}
A_{m,i}^{+,r}=&\{(x,y)\in \ZZ^2, \ x+y=m, \ x\in [0,m-i] \}
\end{align*}
and 
\begin{align*}
A_{m,i}^{-,r}=\cup \{(x,y)\in \ZZ^2, \ -x+y=2i-m, \ x\in [0,m-i] \}. 
\end{align*}
Now we have by symmetry 
\beq
\label{Spatial_10}
P_{z}\left(\tau_{A_m^+}=\sigma_m\right)=\frac{1}{2} P_{(1,j)}\left(\bar\tau_{A_{m,0}^{+,r}}\le\bar\tau_{A_{m,0}^{-,r}},\bar\tau_{A_{m,0}^{+,r}}\le\bar\tau_{C_{m,0}} \right).
\eeq
and
\beq
\label{Spatial_11}
P_{z}\left(\tau_{A_m^-}=\sigma_m\right)=\frac{1}{2} P_{(1,j)}\left(\bar\tau_{A_{m,0}^{-,r}}\le\bar\tau_{A_{m,0}^{+,r}},\bar\tau_{A_{m,0}^{-,r}}\le\bar\tau_{C_{m,0}} \right).
\eeq
The right hand side of the Equation \eqref{Spatial_11} equals to 0 when $j=m-1$. Otherwise, note that if a random walk starting from $(1,j)$ want to visit $A_{m,0}^{-,r}$ before visiting $A_{m,0}^{+,r}$ or $C_{m,0}$, it has to first get through $A_{m,j}^{-,r}$ before visiting $A_{m,0}^{+,r}$ or $C_{m,0}$. Thus
$$
P_{(1,j)}\left(\bar\tau_{A_{m,0}^{-,r}}\le\bar\tau_{A_{m,0}^{+,r}},\bar\tau_{A_{m,0}^{-,r}}\le\bar\tau_{C_{m,0}} \right)\le P_{(1,j)}\left(\bar\tau_{A_{m,j}^{-,r}}\le\bar\tau_{A_{m,0}^{+,r}},\bar\tau_{A_{m,j}^{-,r}}\le\bar\tau_{C_{m,0}} \right). 
$$
Then note that in order to have a random walk starting from $(1,j)$ get to $A_{m,j}^{-,r}$ before visiting $A_{m,0}^{+,r}$ or $C_{m,0}$, it only need to avoid $A_{m,j}^{+,r}$ and $C_{m,j}$. So we have 
$$
P_{(1,j)}\left(\bar\tau_{A_{m,j}^{-,r}}\le\bar\tau_{A_{m,0}^{+,r}},\bar\tau_{A_{m,j}^{-,r}}\le\bar\tau_{C_{m,0}} \right)=P_{(1,j)}\left(\bar\tau_{A_{m,j}^{-,r}}\le\bar\tau_{A_{m,j}^{+,r}},\bar\tau_{A_{m,j}^{-,r}}\le\bar\tau_{C_{m,j}} \right).
$$ 
By symmetry one can see 
$$
P_{(1,j)}\left(\bar\tau_{A_{m,j}^{-,r}}\le\bar\tau_{A_{m,j}^{+,r}},\bar\tau_{A_{m,j}^{-,r}}\le\bar\tau_{C_{m,j}} \right)=P_{(1,j)}\left(\bar\tau_{A_{m,j}^{+,r}}\le\bar\tau_{A_{m,j}^{-,r}},\bar\tau_{A_{m,j}^{+,r}}\le\bar\tau_{C_{m,j}} \right).
$$
Moreover, note that a random walk starting from $(1,j)$ must exist the smaller triangle bounded by $A_{m,j}^{+,r}$, $A_{m,j}^{-,r}$, and $C_{m,j}$ before exiting the larger on bounded by $A_{m,0}^{+,r}$, $A_{m,0}^{-,r}$, and $C_{m,0}$. I.e., 
$$
\sigma^r_j=\bar\tau_{A_{m,j}^{+,r}}\wedge\bar\tau_{A_{m,j}^{-,r}}\wedge \tau_{C_{m,j}}\le \sigma^r=\bar\tau_{A_{m,0}^{+,r}}\wedge\bar\tau_{A_{m,0}^{-,r}}\wedge \tau_{C_{m,0}}.
$$
Thus
\beq
\label{Spatial_12}
\begin{aligned}
P_{(1,j)}&\left(\bar\tau_{A_{m,j}^{+,r}}\le\bar\tau_{A_{m,j}^{-,r}},\bar\tau_{A_{m,j}^{+,r}}\le\bar\tau_{C_{m,j}} \right)\\
&=P_{(1,j)}\left(\bar\tau_{A_{m,j}^{+,r}}\le \sigma^r_j\right)\\
&\le P_{(1,j)}\left(\bar\tau_{A_{m,j}^{+,r}}\le \sigma^r\right)\\
&\le P_{(1,j)}\left(\bar\tau_{A_{m,j}^{+,r}}\le\bar\tau_{A_{m,0}^{-,r}}\wedge \tau_{C_{m,0}} \right)\\
&= P_{(1,j)}\left(\bar\tau_{A_{m,j}^{+,r}}\le\bar\tau_{A_{m,0}^{-,r}},\bar\tau_{A_{m,j}^{+,r}}\le\bar\tau_{C_{m,0}} \right)\\
&\le P_{(1,j)}\left(\bar\tau_{A_{m,0}^{+,r}}\le\bar\tau_{A_{m,0}^{-,r}},\bar\tau_{A_{m,j}^{+,r}}\le\bar\tau_{C_{m,0}} \right).
\end{aligned}
\eeq
Finally note that the right hand side of the last inequality in \eqref{Spatial_12} is exactly the right hand side of \eqref{Spatial_10}. 
\end{proof}
With Lemma \ref{lemma: spatial}, we immediately get \eqref{Spatial_1} from translation invariance. 

\subsection{Proof of Theorem \ref{theorem: maximum_path}}

Now we have all the tools we need to finish the proof of Theorem \ref{theorem: maximum_path}. Recall \eqref{Theorem_1_1} and apply it to $U_n$ and $y_n$, 
$$
\begin{aligned}
&H_{U_n,N}(y_n)\\
&=E_{y_n}\Big[\text{number of visits to $L_N$ in time interval } [0,\tau_{U_n}]\Big]\\
&=\sum_{w\in L_N}P_{y_n}(\tau_{L_N}< \tau_{U_n}, S_{\tau_{L_N}}=w) E_w\Big[\text{number of visits to $L_N$ in time interval} [0,\tau_{U_n}]\Big].
\end{aligned}
$$
Note that for all $w\in L_N$,
$$
P_w(\tau_{L_n}\le \tau_{U_n})=1. 
$$
We have
$$
\begin{aligned}
&H_{U_n,N}(y_n)\\
&\ge \sum_{w\in L_N}P_{y_n}(\tau_{L_N}< \tau_{U_n}, S_{\tau_{L_N}}=w) E_w\Big[\text{number of visits to $L_N$ in time interval} [0,\tau_{L_n}]\Big]\\
&=4P_{y_n}(\tau_{L_N}< \tau_{U_n}) (N-n).
\end{aligned}
$$
Then according again to strong Markov property and the fact that a random walk starting from $y_n$ has to visit $L_{2n}$ before $L_N$, 
$$
P_{y_n}(\tau_{L_N}< \tau_{U_n})=\sum_{w\in L_{2n}} P_{y_n}(\tau_{L_{2n}}< \tau_{U_n}, S_{\tau_{L_{2n}}}=w) P_w(\tau_{L_N}< \tau_{U_n}).
$$
Again, note that for all $w\in L_{2n}$
$$
P_w(\tau_{L_N}< \tau_{U_n})\ge P_w(\tau_{L_N}< \tau_{L_n})=\frac{n}{N-n}.
$$
Thus to prove Theorem \ref{theorem: maximum_path} it is sufficient to show that for $N$ sufficiently larger than $n$, 
\beq
\label{theorem_2_1}
P_{y_n}(\tau_{L_{2n}}< \tau_{U_n}) \ge c n^{-1/2} 
\eeq 
To show \eqref{theorem_2_1}, we have
$$
\begin{aligned}
P_{y_n}(\tau_{L_{2n}}< \tau_{U_n})&\ge \sum_{w\in\mathcal{S}^U_{1,n}} P_{y_n}\left(\tau_{\mathcal{S}_{1,n}}< \tau_{V_n}, S_{\tau_{\mathcal{S}_{1,n}}}=w\in \mathcal{S}^U_{1,n}\right)P_w(\tau_{L_{2n}}< \tau_{U_n})\\
&= \sum_{w\in\mathcal{S}^U_{1,n}} P_{y_n}\left(\tau_{\mathcal{S}_{1,n}}< \tau_{V'_n}, S_{\sigma}=w\in \mathcal{S}^U_{1,n}\right)P_w(\tau_{L_{2n}}< \tau_{U_n}). 
\end{aligned}
$$
Note that by invariance principle there is a constant $c$ such that for any sufficiently large $n$ and $w\in \mathcal{S}^U_{1,n}$, 
$$
P_w(\tau_{L_{2n}}< \tau_{U_n})\ge c. 
$$
Thus
\beq
\label{theorem_2_2}
P_{y_n}(\tau_{L_{2n}}< \tau_{U_n})\ge c P_{y_n}\left(\tau_{\mathcal{S}_{1,n}}< \tau_{V'_n}, S_{\sigma}=w\in \mathcal{S}^U_{1,n}\right).
\eeq
Then by \eqref{escaping probability} and \eqref{Spatial_1}, we have
\beq
\label{theorem_2_3}
P_{y_n}\left(\tau_{\mathcal{S}_{1,n}}< \tau_{V'_n}, S_{\sigma}\in \mathcal{S}^U_{1,n}\right)\ge \frac{1}{2}P_{y_n}\left(\tau_{\mathcal{S}_{1,n}}< \tau_{V'_n}\right)\ge cn^{-1/2}. 
\eeq 
Thus, the proof of Theorem \ref{theorem: maximum_path} is complete. \qed

\section{Total harmonic measure on finite sets}
\label{section: finite}
\subsection{Upper bound in Theorem \ref{theorem: theorem: uniform_2}}
To show the upper bound in \eqref{uniform upper bound_finite_total}, without loss of generality we can assume $B\cap L_0\not=\O$, which implies that $\min_{x\in B}\{x_2\}=0$. Otherwise, for $x_0=(x_{1,0},x_{2,0})$ that has the smallest height in $B$, define
$$ 
B'=B\cup\{(x_{1,0},j),  \ j=0,1,\cdots, x_{2,0}-1\}. 
$$
By Proposition \ref{proposition_finite}, we have $H_{B'}\ge H_B$ and $|B'|=|B|+\min_{x\in B}\{x_2\}$. Thus it suffices for us to prove that for any connected and finite $B$ with $B\cap L_0\not=\O$, 
\beq
\label{finite_upper_bound_1}
H_B\le C|B|. 
\eeq 
And we prove \eqref{finite_upper_bound_1} inductively. When $|B|=1$, we have proved the desired upper bound in \eqref{weak bound}. Suppose we have proved \eqref{finite_upper_bound_1} for all connected $B$ with $|B|\le n$, $B\cap L_0\not=\O$. Then for a $B$ such that $|B|=n+1$, $B\cap L_0\not=\O$, we first show that one can remove one vertex in $B$ and still have a connected subset intersecting $L_0$. In fact we prove something even stronger: 
\begin{lemma}
\label{lemma introduction connected}
For any finite and connected $B\subset \ZZ^2$ with $|B|\ge 2$, there are always two points $x_1,x_2\in B$ such that $B\setminus\{x_1\}$ and $B\setminus\{x_2\}$ are both connected. 
\end{lemma}

\begin{remark}
With Lemma \ref{lemma introduction connected}, we can make sure that starting from $|B|=n+1$, $B\cap L_0\not=\O$, we can remove one point and it will not be in $L_0$ if $|B\cap L_0|=1$. Thus the new connected subset still intersects $L_0$. 
\end{remark}
\begin{proof}
Again, we prove this lemma by induction. For $|B|=2$ or $|B|=3$, it is easy to check the lemma holds. Now suppose it also holds for all connected $|B|\le n$. Then from the assumption we also have that 

\vspace{0.1 in}

\noindent {\bf Observation 1:} for any connected $B$ such that $|B|\le n$ and any $x_0\in B^c$ such that $d(x_0,B)=1$, where 
$$
d(x,B)=\inf_{y\in B}\{|x-y|\}, 
$$
there must exists an $x\in B$ such that $B\setminus\{x\}$ is connected while $d(x_0,B\setminus\{x\})=1$. 

To see this, note that if 
$$
\big|\{y\in B: \ |x_0-y|=1\}\big|\ge 2
$$
then removing one point will not change the distance between $x$ and the smaller subset. So either $x_1$ or $x_2$ in the inductive assumption is good. Otherwise, let $y_0$ be the only point in $B$ neighboring $x_0$. By the inductive assumption we have two points $x_1$ and $x_2$ which we can remove, and one of them must not be $y_0$. Thus we still have an $x\in B$ such that $B\setminus\{x\}$ is connected while $d(x_0,B\setminus\{x\})=1$. 

With the observation above, now for any connected $B$ such that $|B|=n+1$, we first choose one point $y$ arbitrarily from $B$. If $B\setminus\{y\}$ is connected, note that $|B\setminus\{y\}|=n$ and that $d(y,B\setminus\{y\})=1$. Our observation above shows that there must be a $y'\in B\setminus\{y\}$ such that $B\setminus\{y,y'\}$ is also connected and 
$$
d(y,B\setminus\{y,y'\})=1.
$$
This implies that $B\setminus\{y'\}=B\setminus\{y,y'\}\cup \{y\}$ is connected. And we have found our two ``removable" points. Otherwise, if $B\setminus\{y\}$ is not connected, it must have at least two connected components, say $B_1$ and $B_2$. 
\begin{remark}
If $B\setminus\{y\}$ has more than two connected components, just choose two of them arbitrarily. 
\end{remark}
Let 
$$
d(A,B)=\inf_{x\in A, \ y\in B}\{|x-y|\}
$$
for all finite $A$ and $B$. Noting that $B$ is connected, we must have $d(B_1,B\setminus B_1)=1$. But since $B_1$ is not connected to $B\setminus (B_1\cup \{y\})$, we also have 
$$
d\big(B_1,B\setminus (B_1\cup \{y\})\big)\ge 2. 
$$
Thus one can now see $d(y,B_1)=1$ and $d(y,B_2)=1$. Then note that $|B_1|$ and $|B_2|$ are both less than $n$. So by Observation 1 we again have there is a $x_1\in B_1$ such that $B_1\setminus\{x_1\}$ is connected and that 
$$ 
d(y,B_1\setminus\{x_1\})=1,
$$
which implies that $(B_1\setminus\{x_1\})\cup\{y\}$ is connected, 
$$
d\big((B_1\setminus\{x_1\})\cup\{y\}, B_2\big)=1,
$$
and that $B\setminus\{x_1\}=(B_1\setminus\{x_1\})\cup\{y\}\cup B_2$ is connected. The same argument on $B_2$ also gives that there is a $x_2\in B_2$ such that $B\setminus\{x_2\}$ is connected. Finally note that $B_1$ and $B_2$ are different connected component, which implies that $B_1\cap B_2=\O$. So we have $x_1\not=x_2$ and the proof of this lemma is complete. 
\end{proof}

With Lemma \ref{lemma introduction connected}, we continue with the inductive argument for the growth rate of $H_B$. For any finite and connected $B$ such that $|B|=n+1$, $B\cap L_0\not=\O$, there has to be an $x=(x_1,x_2)\in B$ such that $\tilde B=B\setminus \{x\}$ is still connected and  $\tilde B\cap L_0\not=\O$. By inductive assumption we know that $H_{\tilde B}\le Cn$. So now we can concentrate on comparing $H_{\tilde B}$ and $H_{B}$. 

Since $B$ is finite, for any $v\in B$ sufficiently large $N$ we have 
$$
H_{B,N}(v)=\sum_{z\in L_N} P_z(\tau_{v}=\tau_{B\cup L_0}).
$$
And thus 
$$
H_{B,N}=\sum_{v\in B}H_{B,N}(v)=\sum_{z\in L_N} P_z(\tau_{B}\le \tau_{L_0})
$$
while 
$$
H_{\tilde B,N}=\sum_{z\in L_N} P_z(\tau_{\tilde B}\le \tau_{L_0}).
$$
Note that for each $z\in L_N$, 
\beq
\label{inductive_H_1}
P_z(\tau_{B}\le \tau_{L_0})-P_z(\tau_{\tilde B}\le \tau_{L_0})=P_z(\tau_x\le \tau_{L_0}<\tau_{\tilde B}). 
\eeq
Moreover, by strong Markov property, 
\beq
\label{inductive_H_2}
P_z(\tau_x\le \tau_{L_0}<\tau_{\tilde B})=P_z(\tau_x= \tau_{B\cup L_0}) P_x(\tau_{L_0}<\tau_{\tilde B}).
\eeq
Combining \eqref{inductive_H_1} and  \eqref{inductive_H_2}, we have 
\beq
\label{inductive_H_3}
\begin{aligned}
H_{B,N}-H_{\tilde B,N}&=P_x(\tau_{L_0}<\tau_{\tilde B})\sum_{z\in L_N} P_z(\tau_x= \tau_{B\cup L_0})\\
 &=P_x(\tau_{L_0}<\tau_{\tilde B}) H_{B,N}(x). 
 \end{aligned}
\eeq
If $x_2=0$, note that in \eqref{weak bound 2} we have $H_{B,N}(x)\le 1$, which implies that $H_{B,N}-H_{\tilde B,N}\le 1$.  And for $x_2\ge 1$, we have by Theorem \ref{theorem: uniform_path}
\beq
\label{inductive_H_4}
H_{B,N}-H_{\tilde B,N}\le C\sqrt{x_2} P_x(\tau_{L_0}<\tau_{\tilde B}). 
\eeq
And since $\tilde B$ is connected. There must be a finite nearest neighbor path 
$$
\mathcal{\tilde P}_x=\{x=\tilde P_0,\tilde P_1,\tilde P_2,\cdots, \tilde P_{k_n}\in L_0\}
$$
connecting $x$ and $L_0$, where $|\tilde P_i-\tilde P_{i+1}|=1$. And since $d(x,L_0)=x_2$, $|x-P_{k_n}|\ge x_2$. Define
$$
\tilde m_x=\inf\{i: |\tilde P_i-x|\ge x_2\}
$$
and 
$$
\mathcal{\tilde Q}_x=\{\tilde P_1,\tilde P_2,\cdots, \tilde P_{\tilde m_x}\}.
$$
One can immediately see that $\mathcal{\tilde Q}_x\subset \tilde B$ and that 
$$
\mathcal{\tilde Q}_x\subset B(x,2x_2). 
$$
Thus
$$
\frac{1}{4}P_x(\tau_{L_0}<\tau_{\tilde B})\le \frac{1}{4}P_{x}(\tau_{L_0}<\tau_{\mathcal{\tilde Q}_x})\le P_{\tilde P_1}(\tau_{L_0}<\tau_{\mathcal{\tilde Q}_x}). 
$$
Then for sufficiently large $x_2$, recalling the $C_2$ defined in Lemma \ref{lemma: inverse_4}, let $r=x_2/2C_2$ and $D_x=\mathcal{\tilde Q}_x\cap B(\tilde P_1,r)$. Note that for sufficiently large $x_2$, a random walk starting at $\tilde P_1$ has to hit $\partial^{out}B(\tilde P_1,C_2\cdot r)$ before reaching $L_0$. We have by \eqref{basic 1}, \eqref{basic 2} and Lemma \ref{lemma: 2.15},
\beq
\label{inductive_H_5}
\begin{aligned}
P_{\tilde P_1}(\tau_{L_0}<\tau_{\mathcal{\tilde Q}_x})&\le P_{\tilde P_1}(\tau_{\partial^{out}B(x,C_2\cdot r)}<\tau_{\mathcal{\tilde Q}_x\cap B(x,C_2\cdot r)})\\
&\le P_{\tilde P_1}(\tau_{\partial^{out}B(x,C_2\cdot r)}<\tau_{D_x})\\
&\le c_3^{-1}\mu_{D_x-\tilde P_1}(0)\le Cx_2^{-1/2}. 
\end{aligned}
\eeq
Combining \eqref{inductive_H_4} and \eqref{inductive_H_5} we have that there is a constant $C$ independent to $n,N$ and $x$, such that 
$$
H_{B,N}-H_{\tilde B,N}\le C.
$$
Thus the proof of \eqref{uniform upper bound_finite_total} is complete. 
 
\subsection{Lower bound in Theorem \ref{theorem: theorem: uniform_2}}
First, \eqref{uniform lower bound_finite_total_1} is obvious. Now we show the lower bound in \eqref{uniform lower bound_finite_total}. Since $B$ is finite, let $\bar x=(\bar x_1, \bar x_2)$ be a point in $B$ such that 
$$
\bar x_2=\max_{x\in B}\{x_2\}.
$$
Note that by Proposition \ref{proposition_finite}, $H_B\ge H_{\bar x}$. It suffices to prove \eqref{uniform lower bound_finite_total} for the single element set $\{\bar x\}$. Recall that 
$$
\begin{aligned}
H_{\{\bar x\}, N}&=E_{\bar x}\Big[\text{number of visits to $L_N$ in }[0,\tau_{\{\bar x\}\cup L_0}]\Big]\\
&\ge P_{\bar x}(\tau_{L_N}<\tau_{\{\bar x\}\cup L_0}) \inf_{z\in L_N} E_{z}\Big[\text{number of visits to $L_N$ in }[0,\tau_{\{\bar x\}\cup L_0}]\Big]
\end{aligned}
$$
and that for sufficiently large $N$ and any $z\in L_N$, 
$$
\begin{aligned}
&E_{z}\Big[\text{number of visits to $L_N$ in }[0,\tau_{\{\bar x\}\cup L_0}]\Big]\\
\ge &E_{z}\Big[\text{number of visits to $L_N$ in }[0,\tau_{L_{x_2}}]\Big]=4(N-x_2)\ge 2N. 
\end{aligned}
$$
To prove \eqref{uniform lower bound_finite_total} it is sufficient to show that for sufficiently large $x_2$
\beq
\label{lower_finite_1}
P_{\bar x}(\tau_{L_N}<\tau_{\{\bar x\}\cup L_0}) \ge \frac{c \bar x_2}{N \log(\bar x_2)}. 
\eeq
Now let $n_{\bar x}$ be the largest odd number less than $\bar x_2$. We define $B_1(\bar x,n_{\bar x})$ be the $L_1$ ball centered at $x$ with radius $n_x$. Moreover we define 
$$
W^{1,+}_{\bar x}=\partial B_1(\bar x,n_{\bar x})\cap \{(y_1,y_2), \ y_2\ge \bar x_2+(n_{\bar x}+1)/2\}
$$
be the upper corner of $\partial B_1(\bar x,n_{\bar x})$. By symmetry we have 
$$
P_{\bar x}\left(\tau_{\partial B_1(\bar x,n_{\bar x})}<\tau_{\bar x},  \ S_{\tau_{\partial B_1(\bar x,n_{\bar x})}}\in W^{1,+}_{\bar x}\right)=\frac{1}{4}P_{\bar x}(\tau_{\partial B_1(\bar x,n_{\bar x})}<\tau_{\bar x}).
$$
Thus we have 
\beq
\label{lower_finite_2}
\begin{aligned}
P_{\bar x}&(\tau_{L_N}<\tau_{\{\bar x\}\cup L_0}) \\
&\ge P_{\bar x}\left(\tau_{\partial B_1(\bar x,n_{\bar x})}<\tau_{\bar x},  \ S_{\tau_{\partial B_1(\bar x,n_{\bar x})}}\in W^{1,+}_{\bar x}\right) \inf_{y\in W^{1,+}_{\bar x}} P_y(\tau_{L_N}<\tau_{L_{\bar x_2}})\\
&\ge \frac{1}{4}P_{\bar x}(\tau_{\partial B_1(\bar x,n_{\bar x})}<\tau_{\bar x}) \inf_{y\in W^{1,+}_{\bar x}} P_y(\tau_{L_N}<\tau_{L_{\bar x_2}}).
\end{aligned}
\eeq
Then note that for any $y\in W^{1,+}_{\bar x}$
$$
 P_y(\tau_{L_N}<\tau_{L_{\bar x_2}})\ge \frac{n_{\bar x}}{2N}\ge \frac{\bar x_2}{4N}. 
$$
Thus it is sufficient for us to prove that 
\begin{lemma}
\label{lemma:lower_finite}
There is a constant $c>0$ such that for all sufficiently large $\bar x_2$, 
\beq
\label{lower_finite_2}
P_{\bar x}(\tau_{\partial B_1(\bar x,n_{\bar x})}<\tau_{\bar x})\ge \frac{c}{\log(\bar x_2)}. 
\eeq
\end{lemma}
\begin{proof}
For $S_n=(S_{1,n},S_{2,n})$ to be the simple random walk starting at $\bar x$, consider the martingale 
$$
M_n=(S_{2,n}-\bar x_2)^2-\frac{n}{2}. 
$$
Note that $M_0=0$, so we have 
$$
E_{\bar x}[\tau_{\partial B_1(\bar x,n_{\bar x})}]\le \sup_{y\in \partial B_1(\bar x,n_{\bar x})} (y_2-\bar x_2)^2\le \bar x_2^2.
$$
Thus 
\beq
\label{lower_finite_3}
P_{\bar x}(\tau_{\partial B_1(\bar x,n_{\bar x})}\ge x_2^3)\le \frac{1}{x_2}. 
\eeq
On the other hand, for simple random walk in $\ZZ^2$ it was shown in \cite{excursion2} and \cite{excursion22} that for sufficiently large $x_2$, 
\beq
\label{excursion simple 2}
\begin{aligned}
P_{\bar x}(\tau_{\bar x}>\bar x_2^3)&=\frac{\pi}{\log(\bar x^3_2)}+O\left(\frac{1}{\log^2(\bar x^3_2)}\right)\ge \frac{\pi}{6\log(\bar x_2)}.
\end{aligned}
\eeq
Thus note that 
$$
\begin{aligned}
P_{\bar x}(\tau_{\partial B_1(\bar x,n_{\bar x})}<\tau_{\bar x})&\ge P_{\bar x}(\tau_{\partial B_1(\bar x,n_{\bar x})}<\bar x_2^3, \ \tau_{\bar x}>\bar x_2^3)\\
&=P_{\bar x}(\tau_{\bar x}>\bar x_2^3)-P_{\bar x}(\tau_{\partial B_1(\bar x,n_{\bar x})}\ge \bar x_2^3, \ \tau_{\bar x}>\bar x_2^3)\\
&\ge P_{\bar x}(\tau_{\bar x}>\bar x_2^3)-P_{\bar x}(\tau_{\partial B_1(\bar x,n_{\bar x})}\ge \bar x_2^3). 
\end{aligned}
$$
Combining \eqref{lower_finite_3} and \eqref{excursion simple 2}, we have for sufficiently large $x_2$,
$$
P_{\bar x}(\tau_{\partial B_1(\bar x,n_{\bar x})}<\tau_{\bar x})\ge \frac{\pi}{6\log(\bar x_2)}-\frac{1}{x_2}\ge \frac{\pi}{7\log(\bar x_2)}
$$
which finished the proof of this lemma. 
\end{proof}
With Lemma \ref{lemma:lower_finite}, the proof of  \eqref{uniform lower bound_finite_total} and thus Theorem \ref{theorem: theorem: uniform_2} is complete. \qed

\subsection{Proof of Theorem \ref{theorem: theorem: maximum_2}}

Now we show that the total harmonic measure is maximized (up to multiplying a constant) by the vertical line segment $V_n$ over all connected finite subsets with the same cardinality and connected to $L_0$. And again we do this inductively. By \eqref{inductive_H_3}, we have
\beq
\label{maximum_finite}
H_{V_{n},N}-H_{V_{n-1},N}=H_{V_{n},N}(y_{n}) P_{y_{n}}(\tau_{L_0}<\tau_{V_{n-1}}). 
\eeq
According to Theorem \ref{theorem: maximum_path}, we have 
$$
H_{V_{n},N}(y_{n})\ge x\sqrt{n}.
$$
Noting that 
$$
P_{y_{n}}(\tau_{L_0}<\tau_{V_{n-1}})\ge P_{y_{n}}(\tau_{L_0}<\tau_{V_{n}}),
$$
it suffices to prove that 
\beq
\label{maximum_finite_2}
P_{y_{n}}(\tau_{L_0}<\tau_{V_{n}})\ge \frac{c}{\sqrt{n}}. 
\eeq
On the other hand, recall that 
$$
\mathcal{S}_{1,n}=\partial B_1(y_n,[n/3])
$$
and that 
$$
\mathcal{S}^U_{1,n}=\mathcal{S}_{1,n}\cap\{(x,y), y\ge n\}. 
$$
We have 
$$
P_{y_{n}}(\tau_{L_0}<\tau_{V_{n}})\ge P_{y_{n}}\Big(\tau_{\mathcal{S}^U_{1,n}}<\tau_{V_{n}}\Big) \inf_{y\in \mathcal{S}^U_{1,n}}P_y(\tau_{L_0}<\tau_{V_{n}}).
$$
Again by invariant principle, there is a constant $c>0$ such that for any $n$ and $y\in \mathcal{S}^U_{1,n}$, 
$$
P_y(\tau_{L_0}<\tau_{V_{n}})\ge c. 
$$
And then by \eqref{escaping probability} and \eqref{Spatial_1}, 
$$
P_{y_{n}}\Big(\tau_{\mathcal{S}^U_{1,n}}<\tau_{V_{n}}\Big) \ge \frac{c}{\sqrt{n}}.
$$
Thus the proof of Theorem \ref{theorem: theorem: maximum_2} is complete. \qed

\section{Construction and growth estimate of DLA in $\mathcal{H}$}
\label{section: particle system}

\subsection{Construction of a growth model}

With the upper bounds of the harmonic measure on the upper half plane, in this section we construct pure growth models which can be used as a dominating process for both the DLA model in $\mathcal{H}$ and the stationary DLA model that will be introduced in a follow up paper. Consider an interacting particle system $\bar \xi_t$ defined on $\{0,1\}^\mathcal{H}$ where $\mathcal{H}$ is the upper half plane with 1 standing for a site occupied while 0 for vacant, with transition rates as follows: 
\begin{enumerate}[(i)]
\item For each occupied site $x=(x_1,x_2)\in \mathcal{H}$, if $x_2>0$ it will try to give birth to each of its nearest neighbors at a Poisson rate of $\sqrt{x_2}$. If $x_2=0$, it will try to give birth to each of its nearest neighbors at a Poisson rate of $1$. 
\item When $x$ attempts to give birth to its nearest neighbors $y$ already occupied, the birth is suppressed.   
\end{enumerate}
We prove that an interacting particle system determined by the dynamic above is well-defined. 
\begin{Proposition}
\label{lemma_well_definition_1}
The interacting particle system $\bar \xi_t\in \{0,1\}^\mathcal{H}$ satisfying (i) and (ii) is well defined.  
\end{Proposition}

\begin{proof}
The proof of Proposition \ref{lemma_well_definition_1} uses a similar idea as in Theorem 2.1 of \cite{ten_lectures}. Although here the transition rates are no longer translation invariant or uniformly bounded, we will be able to use more elaborate argument and show that with high probability the time moving forward at each step goes to 0 while still being un-summable all together. The next idea is very similar to Borel-Cantelli lemma. However, rather than using the result directly, we will have the proof of Borel-Cantelli lemma embedded in our argument. By doing so, we will be able to make sure the space-time box in each step of our iteration is deterministic and can be explicitly calculated. 

Our construction starts with introducing the following families of independent Poisson processes: for all $x=(x_1,x_2)$ and $y=(y_1,y_2)$ that are nearest neighbors in $\mathcal{H}$ and $e_{x\to y}$ which is the oriented edge from $x$ to $y$, let 
$$
\big\{N_t^{x\to y},\ x,y\in \mathcal{H},\ \|x-y\|=1\big\}
$$ 
be a family of independent Poisson processes, where $N_t^{x\to y}$ has intensity $\sqrt{x_2}\vee 1$. Then let 
$$
\big\{\tilde N_t^{x\to y},\ x,y\in \mathcal{H}, \ \|x-y\|=1\big\}
$$ 
be a family of independent Poisson process independent to $N_t$ with the same intensities. Now consider the space-time combination, $\mathcal{H}\times (-\infty,\infty)$. From each $x\in \mathcal{H}$, we draw a vertical infinite line to represent the double infinite time line at this site. Then for each $e_{x\to y}$, at any time $t$ such that $N_t^{x\to y}=N_{t-}^{x\to y}+1$, we draw an oriented arrow from $(x,t)$ to $(y,t)$. And at $t$ such that $\tilde N_t^{x\to y}=\tilde N_{t-}^{x\to y}+1$, we draw an oriented arrow from $(x,-t)$ to $(y,-t)$.
\begin{remark}
Although the construction of our particle system actually only depend on the transitions on the positive time line, by defining the transition for negative $t$'s we are able to have better symmetry on the time reversal and thus formally simplify the proof. 
\end{remark}
We have an oriented random graph in the space-time combination. Then for any two points $(x,t_1)$ and $(x',t_2)$ with $t_1<t_2$, we define that $(x,t_1)$ and $(x',t_2)$ are connected or $(x,t_1)\to (x',t_2)$, if there is a (finite) path in the oriented random graph starting from $(x,t_1)$, that goes up vertically and follows the oriented edges ending at $(x',t_2)$. Then 
\begin{definition}
\label{definition_upper_bound_system}
For any $\bar\xi_0\in \{0,1\}^\mathcal{H}$, we define $\bar \xi_t$ such that for each $t\ge 0$ and $x\in \mathcal{H}$, $\bar \xi_t(x)=1$ if and only if there is a $x'$ such that $\bar\xi_0(x')=1$ and $(x',0)\to(x,t)$.     
\end{definition}

Once we prove that $\bar\xi_t$ is well defined, one can check that the conditions (i), (ii) are statisfied.
And to show that $\bar \xi_t$ is well defined, it suffices to prove that in our  oriented random graph, with probability one $(x,t)$ can be connected to only finitely many points $(x',0)$ so one can determine explicitly whether any of them is occupied in the initial configuration. To be precise, for any $x\in \mathcal{H}$ and any $t, T\ge 0$, define subset
\beq
\label{range_of_ancestors}
R_{t,T}(x)=\{y\in \mathcal{H}, \ \text{s.t. } (y,T-t)\to (x,T) \}
\eeq
be the set of all possible ancestors of $\bar \xi_T(x)$ at time $T-t$, and we will write $R_{T}(x)$ in short of $R_{T,T}(x)$. According to the definition, it is easy to see that 
\beq
R_{t_1,T}(x)\subset R_{t_2,T}(x)
\eeq
for all $0\le t_1\le t_2$ and $T\ge 0$, and that 
\beq
\label{monotone_1}
R_{T_1}(x)\subset R_{T_2}(x)
\eeq
for any $0\le T_1\le T_2$. Thus, to show that Definition \ref{definition_upper_bound_system} is self-consistent, we only need to prove that 
\begin{lemma}
\label{lemma_well_definition_2}
With probability one we have $R_{T}(x)$ is finite for any $T>0$ and $x\in \mathcal{H}$.  
\end{lemma}
\begin{proof}
Let
$$
\text{Rad}_{t,T}(x)=\sup_{y\in R_{t,T}(x)}|x-y|
$$
be the radius of $R_{t,T}(x)$ and $\text{Rad}_{T}(x)=\text{Rad}_{T,T}(x)$. By \eqref{monotone_1}, it is sufficient to prove that for each given $T>0$ and $x\in \mathcal{H}$ we have 
\beq
\text{Rad}_{T}(x)<\infty
\eeq
almost surely. Then, we can take all rational numbers of $T$'s and all $x\in \mathcal{H}$ which are both countable to get the lemma. Moreover, note that to show $P(\text{Rad}_{T}(x)<\infty)=1$, it suffices to prove that for any $\ep>0$,
\beq
P(\text{Rad}_{T}(x)<\infty)>1-\ep. 
\eeq
For any given $T$ and $t\ge 0$ and $x=(x^{(1)},x^{(2)})\in \mathcal{H}$, note that $R_{t,T}(x)$ is the collection of all $x'$ such that $(x',T-t)$ is connected to $(x,T)$. And for $(x',T-t)$ and $(x,T)$ to be connected, there must be a path between them, i.e., there must be a sequence of times $T-t\le t_1<t_2<\cdots<t_n\le T$ and $x'=x_0, x_1,x_2,\cdots, x_n=x$ which is a nearest neighbor sequence in $\mathcal{H}$ such that 
$$
N^{x_{i-1}\to x_{i}}_{t_i}=N^{x_{i-1}\to x_{i}}_{t_i-}+1
$$
if $t_i\ge0$, or
$$
\tilde N^{x_{i-1}\to x_{i}}_{-t_i}=\tilde N^{x_{i-1}\to x_{i}}_{-t_i-}+1
$$
if $t_i<0$, for all $i=1,2,\cdots, n$. Thus it is easy to see that for any nearest neighbor  path $x_0, x_1,x_2,\cdots, x_n=x$ in $\mathcal{H}$, it is open between $T-t$ and $T$ in our oriented random graph only if there is  at least one transition at each edge along the path during this time interval.  Thus we have
\beq
\label{probability_range_0}
\begin{aligned}
P&(\text{Rad}_{t,T}(x)\ge n)\\
&\le P(\exists \text{ an open path in $[T-t,T]$ ending at $x$ with length $n$})\\
&\le \sum_{x_0, x_1,x_2,\cdots, x_n\in \mathcal{P}_{x,n}} \prod_{i=1}^n P(N^{x_{i-1}\to x_{i}}_T-N^{x_{i-1}\to x_{i}}_{T-t}\ge 1)\\
&= \sum_{x_0, x_1,x_2,\cdots, x_n\in \mathcal{P}_{x,n}} \prod_{i=1}^n \left[1- e^{-t\sqrt{x^{(2)}_{i-1}}}\right]\\
&\le t^n\sum_{x_0, x_1,x_2,\cdots, x_n\in \mathcal{P}_{x,n}} \sqrt{\prod_{i=1}^n x^{(2)}_{i-1}}
\end{aligned}
\eeq
where $\mathcal{P}_{x,n}$ is the collection of all nearest neighbor paths in $\mathcal{H}$ of length $n$ ending at $x$, and $x^{(2)}_{i-1}$ stands for the $y-$coordinate of $x_{i-1}$. 
\begin{remark}
Without loss of generality, the inequalities above is written for $0\le t\le T$. By symmetry the same hold for $t>0$ and $T<0$. Note that even when $T>0$ and $T-t<0$, the total number of transitions of $N^{x_{i-1}\to x_{i}}_s$ in $s\in[0,T]$ plus the  total number of transitions of $\tilde N^{x_{i-1}\to x_{i}}_s$ in $s\in[0,t-T]$ is again a Poisson random variable with intensity $t\sqrt{x^{(2)}_{i-1}}$. So \eqref{probability_range_0} still holds. 
\end{remark}
Then note that $|\mathcal{P}_{x,n}|\le 4^n$ and that for each $x_0, x_1,x_2,\cdots, x_n=x\in \mathcal{P}_{x,n}$, we have 
$$
x^{(2)}_{n-i}\le x^{(2)}+i, \ i=0,1,2,\cdots, n. 
$$
Thus, we have 
$$
 \sqrt{\prod_{i=1}^n x^{(2)}_{i-1}}\le  \sqrt{\prod_{i=1}^n (x^{(2)}+i)}
$$
which implies that 
\beq
\label{probability_range}
P(\text{Rad}_{t,T}(x)\ge n)\le (4t)^n \sqrt{\prod_{i=1}^n (x^{(2)}+i)}.
\eeq
Now for each $\gamma\in (0,1/2)$, define 
$$
M_\gamma=\sum_{k=0}^\infty k^{2/(1-\gamma)} 2^{-k^{\gamma/(1-\gamma)}}<\infty.
$$
Now for any $\ep>0$ let 
$$
N_1=\left\lfloor\frac{4 \big(x^{(2)}\big)^\gamma }{1-\gamma}\right\rfloor
$$ 
and 
$$
\delta_1=t_1=\frac{\ep}{64M_\gamma \sqrt{x^{(2)}+N_1}}.
$$
By \eqref{probability_range}, we have
\beq
\label{probability_range_1}
P(\text{Rad}_{t_1,T}(x)\ge N_1)\le  \frac{\ep}{M_\gamma} 2^{-N_1}\le   \frac{\ep}{M_\gamma} k_1^{2/(1-\gamma)} 2^{-k_1^{\gamma/(1-\gamma)}}
\eeq
where
\beq
\label{probability_range_2}
k_1=\lfloor\big(x^{(2)}\big)^{1-\gamma}\rfloor.
\eeq
The last inequality in \eqref{probability_range_1} is a result that
$$
N_1=\left\lfloor\frac{4 \big(x^{(2)}\big)^\gamma }{1-\gamma}\right\rfloor\ge \big(x^{(2)}\big)^\gamma\ge k_1^{\gamma/(1-\gamma)}. 
$$
Then under event 
$$
E_1=A_1=\{\text{Rad}_{t_1,T}(x)< N_1\},
$$
we define 
$$
x^{(2),2}=x^{(2)}+N_1,
$$
$$
N_2=\left\lfloor\frac{4 \big(x^{(2),2}\big)^\gamma }{1-\gamma}\right\rfloor,
$$
$$
\delta_2=\frac{\ep}{64M_\gamma \sqrt{x^{(2),2}+N_2}},
$$
and
$$
t_2=t_1+\delta_2.
$$
Then define event 
$$
A_2=\bigcap_{y\in B(x,x^{(2),2}-x^{(2)}-1)}\{\text{Rad}_{\delta_2,T-t_1}(y)< N_2\}.
$$ 
One can first see that by the same calculation as in \eqref{probability_range_1}
\beq
\label{probability_range_3}
\begin{aligned}
P(A_2)&\ge 1-\sum_{y\in B(x,x^{(2),2}-x^{(2)}-1)} P(\text{Rad}_{\delta_2,T-t_1}(y)\ge N_2)\\
&\ge 1-\big(x^{(2),2}\big)^{2}\frac{\ep}{16M_\gamma}2^{-k_2^{\gamma/(1-\gamma)}}
\end{aligned}
\eeq
where
$$
k_2=\lfloor\big(x^{(2),2}\big)^{1-\gamma}\rfloor.
$$
Moreover, we have 
$$
\big(x^{(2),2}\big)^{2}=\left[(x^{(2),2}\big)^{1-\gamma} \right]^{2/(1-\gamma)}
$$
while
$$
(x^{(2),2}\big)^{1-\gamma} \le 2k_2.
$$
Thus
\beq
\label{probability_range_3}
\begin{aligned}
P(A_2)&\ge1-\big(2k_2\big)^{2/(1-\gamma)}\frac{\ep}{16M_\gamma}2^{-k_2^{\gamma/(1-\gamma)}}\\
&\ge 1-\frac{\ep}{M_\gamma}k_2^{2/(1-\gamma)}2^{-k_2^{\gamma/(1-\gamma)}}.
\end{aligned}
\eeq
Then note that for any $x\ge 1$, we have by calculus 
$$
\big(x^{1-\gamma}+4\big)^{1/(1-\gamma)}> x+\frac{4}{1-\gamma} x^\gamma
$$
while 
$$
\begin{aligned}
\big(x^{1-\gamma}+1\big)^{1/(1-\gamma)}&\le x+\frac{1}{1-\gamma} \big(x^{1-\gamma}+1\big)^{\gamma/(1-\gamma)}\\
&\le x+\frac{1}{1-\gamma} \big(2x^{1-\gamma}\big)^{\gamma/(1-\gamma)}\\
&< x+\frac{2}{1-\gamma} x^\gamma\\
&< x+\left\lfloor\frac{4 x^\gamma }{1-\gamma}\right\rfloor.
\end{aligned}
$$
We have that 
\beq
\label{probability_range_4}
\big(x^{(2),2}\big)^{1-\gamma}=\big(x^{(2)}+N_1\big)^{1-\gamma}\in \Big(\big(x^{(2)}\big)^{1-\gamma}+1,\big(x^{(2)}\big)^{1-\gamma}+4 \Big)
\eeq
and that
$$
k_2=\lfloor\big(x^{(2),2}\big)^{1-\gamma}\rfloor\in [k_1+1,k_1+4). 
$$
Using exactly the same argument on 
$$
x^{(2),3}=x^{(2),2}+N_2,
$$
and 
$$
k_3=\lfloor\big(x^{(2),3}\big)^{1-\gamma}\rfloor,
$$
we have 
$$
k_3\in [k_2+1,k_2+4). 
$$
Then we note that event $A_1$ depends only on the transitions within $B(x,N_1)\times [T-t_1,T]$, while event $A_2$ depends only on the transitions within $B(x,N_1+N_2)\times [T-t_2,T-t_1]$. By the independence of increment in a Poisson processes, we have that $A_1$ independent to $A_2$, and thus for $E_2=A_1\cap A_2$,
$$
P(E_2)=P(A_1)P(A_2)\ge \left( 1-\frac{\ep}{M_\gamma}k_1^{2/(1-\gamma)}2^{-k_1^{\gamma/(1-\gamma)}}\right)\cdot\left( 1-\frac{\ep}{M_\gamma}k_2^{2/(1-\gamma)}2^{-k_2^{\gamma/(1-\gamma)}}\right).
$$
Finally, recalling the definition of $\text{Rad}_{t_1,T}$, one can immediately have under $E_2$
$$
\text{Rad}_{t_2,T}(x)<x^{(2),3}-x^{(2)}<\infty. 
$$
Repeat such iteration, i.e., for all $n\ge 2$ let 
$$
x^{(2),n}=x^{(2),n-1}+N_{n-1},
$$
$$
N_n=\left\lfloor\frac{4 \big(x^{(2),n}\big)^\gamma }{1-\gamma}\right\rfloor,
$$
$$
\delta_n=\frac{\ep}{64M_\gamma \sqrt{x^{(2),n}+N_n}},
$$
$$
t_n=t_{n-1}+\delta_n,
$$
$$
A_n=\bigcap_{y\in B(x,x^{(2),n}-x^{(2)}-1)}\{\text{Rad}_{\delta_n,T-t_{n-1}}(y)< N_n\},
$$ 
and
$$
E_n=E_{n-1}\cap A_n.
$$
Consider 
$$
E_\infty=\bigcap_{n=1}^\infty A_n.
$$
Under $E_\infty$ we have for any $n\ge1$, 
\beq
\label{probability_range_5}
\text{Rad}_{t_n,T}(x)<x^{(2),n+1}-x^{(2)}<\infty. 
\eeq
At the same time
\beq
\label{probability_range_6}
\begin{aligned}
t_n&=\sum_{i=1}^n\delta_i=\sum_{i=1}^n \frac{\ep}{64M_\gamma \sqrt{x^{(2),i}+N_i}}=\frac{\ep}{64M_\gamma}\sum_{i=1}^n \frac{1}{ \sqrt{x^{(2),i+1}}}.
\end{aligned}
\eeq
Moreover, by \eqref{probability_range_4} we have for each $i$
\beq
\label{probability_range_7}
\big(x^{(2),i}\big)^{1-\gamma}=\big(x^{(2),i-1}+N_{i-1}\big)^{1-\gamma}\in \big(\big(x^{(2),i-1}\big)^{1-\gamma}+1,\big(x^{(2),i-1}\big)^{1-\gamma}+4 \big)
\eeq
which together implies that 
\beq
\label{probability_range_8}
\big(x^{(2),i}\big)^{1-\gamma}\le \big(x^{(2)}\big)^{1-\gamma}+4i. 
\eeq
Combining \eqref{probability_range_6}-\eqref{probability_range_8} we have 
\beq
\label{probability_range_9}
\begin{aligned}
t_n\ge\frac{\ep}{64M_\gamma}\sum_{i=1}^n \left[\big(x^{(2)}\big)^{1-\gamma}+4i\right]^{-1/(2-2\gamma)}.
\end{aligned}
\eeq
Recalling that $\gamma\in (0,1/2)$, $1/(2-2\gamma)<1$, which implies that the series in \eqref{probability_range_9} is divergent. So for any $T\ge 0$ there is a $n(T,\gamma,\ep)<\infty$ such that for all $n\ge n(T,\gamma,\ep)$, $t_n\ge T$, 
$$
\text{Rad}_{T}(x)\subset\text{Rad}_{t_n,T}(x).
$$
Thus we have under event $E_\infty$, $\text{Rad}_{T}(x)<\infty$. On the other hand, Noting that by the independence increment of Poisson processes, we have $A_1, A_2,\cdots$ gives a sequence of independent events, and that according to our iteration for each $i$
\beq
\label{probability_range_10}
\begin{aligned}
P(A_i)\ge 1-\frac{\ep}{M_\gamma}k_i^{2/(1-\gamma)}2^{-k_i^{\gamma/(1-\gamma)}}>1-\ep
\end{aligned}
\eeq
with 
\beq
\label{probability_range_11}
k_i\in [k_{i-1}+1,k_{i-1}+4). 
\eeq
Thus for sufficiently small $\ep$ such that for all $x\in (0,\ep)$, $\log(1-x)\ge -2x$ and any $n\ge 1$, we have
$$
P(E_n)-1\ge\log\big( P(E_n)\big)=\log\left( \prod_{i=1}^n P(A_i)\right)\ge 2\sum_{i=1}^n[P(A_i)-1]. 
$$
By \eqref{probability_range_10}, 
$$
P(E_n)-1\ge -\frac{2\ep}{M_\gamma}\sum_{i=1}^n k_i^{2/(1-\gamma)}2^{-k_i^{\gamma/(1-\gamma)}}.
$$
Then noting that by \eqref{probability_range_11} $k_i\ge k_{i-1}+1$ and the fact that all $k_i$'s are integers by definition, we have
$$
\sum_{i=1}^n k_i^{2/(1-\gamma)}2^{-k_i^{\gamma/(1-\gamma)}}\le M_\gamma
$$
and thus
\beq
\label{probability_range_12}
P(E_n)\ge 1-2\ep. 
\eeq
Note that the right hand side of \eqref{probability_range_12} is independent of $n$. We have $P(E_\infty)\ge 1-2\ep$. And since $\ep$ is arbitrarily chosen,  $P(\text{Rad}_T(x)<\infty)=1$ which completes the  proof of Lemma \ref{lemma_well_definition_2}. 
\end{proof}
Thus the proof of Proposition \ref{lemma_well_definition_1} is complete. 
\end{proof}

With the proof of Proposition \ref{lemma_well_definition_1}, one can easily apply the technique of Poisson thinning to define the following particle system where time is slowed down in-homogeneously and define a dominating process for the future stationary DLA model. I.e.,  we can consider the slower interacting particle system $\tilde \xi_t$ defined on $\{0,1\}^\mathcal{H}$ with transition rates as follows: 
\begin{enumerate}[(i)']
\item For each occupied site $x=(x_1,x_2)\in \mathcal{H}$ at time $t\ge 0$, it will try to give birth to each of its nearest neighbors at a Poisson rate of $\frac{\sqrt{x_2}}{\sqrt{t+1}}$. 
\item When $x$ attempts to give birth to its nearest neighbors $y$ already occupied, the birth is suppressed.   
\end{enumerate}
For $\tilde \xi_t$ we have 
\begin{corollary}
\label{lemma_well_definition_3}
The interacting particle system $\tilde \xi_t\in \{0,1\}^\mathcal{H}$ satisfying (i)' and (ii)' is well defined.  
\end{corollary}

\begin{proof}
We construct $\tilde \xi_t$ with the same families of Poisson processes. Recall that in the proof of Proposition \ref{lemma_well_definition_1}, for all $x=(x_1,x_2)$ and $y=(y_1,y_2)$ in $\mathcal{H}$ with $|x-y|=1$ and $e_{x\to y}$ which is the the oriented edge from $x$ to $y$, we have 
$$
\big\{N_t^{x\to y},\ x,y\in \mathcal{H}, \ |x-y|=1\big\}
$$ 
be a family of independent Poisson process with intensity of $N_t^{x\to y}$ equals to $\sqrt{x_2}$. Moreover, for each $e_{x\to y}$, we define $\{U^{x\to y}_n\}_{n=1}^\infty$ be a i.i.d. sequence of random variables uniform on $[0,1]$. And we let the sequences for different edges independent to each other and also independent to the Poisson processes previously defined. 

Now consider the space-time combination, $\mathcal{H}\times [0,\infty)$. From each $x\in \mathcal{H}$, we draw a vertical infinite half line to represent the time line at this site. Then for each $e_{x\to y}$, at any time $t$ such that $N_t^{x\to y}=n=N_{t-}^{x\to y}+1$, we draw an oriented arrow from $(x,t)$ to $(y,t)$ if $U^{x\to y}_n<1/\sqrt{t+1}$. Thus we have another oriented random graph in the space-time combination which is a subset of the one we have for $\bar \xi_t$. By Proposition \ref{lemma_well_definition_1} we can see the following particle system is well defined. 
\begin{definition}
\label{definition_upper_bound_system_2}
For any $\tilde\xi_0\in \{0,1\}^\mathcal{H}$, we define $\tilde \xi_t$ such that for each $t\ge 0$ and $x\in \mathcal{H}$, $\tilde \xi_t(x)=1$ if and only if there is a $x'$ such that $\bar\xi_0(x')=1$ and $(x',0)$ is connected to $(x,t)$ in the new smaller oriented random graph.     
\end{definition}
\end{proof}

\subsection{Proof of Theorem \ref{theorem: DLA_1}}
 
By Theorem \ref{theorem: uniform_path} we have seen that for any $B$, $x\in B\setminus L_0$ and any $\vec e=x\to y$ with $\|x-y\|=1$,
$$
H_{B}(\vec e)\le H_{B}(x)\le C\sqrt{x_2}
$$
for some $C>1$. Moreover, by \eqref{weak bound 2}, if $x_2=0$, 
$$
H_{B}(\vec e)\le H_{B}(x)\le 1. 
$$
We construct our DLA model on $\mathcal{H}$ as follows: First, recall that 
$$
\big\{N_t^{x\to y},\ x,y\in \mathcal{H}, \ \|x-y\|=1\big\}
$$ 
is a family of independent Poisson processes, such that the intensity of $N_t^{x\to y}$ equals to $\sqrt{x_2}$ and that $\{U^{x\to y}_n\}_{n=1}^\infty$ is an i.i.d. sequence of random variables uniform on $[0,1]$ independent to the Poisson processes. Let $\bar A_0=\{(0,0)\}$, and for any $t>0$,
\begin{itemize}
\item If there is an $\vec e=x\to y$ such that $x\in \bar A_{t-}$ and $y\notin \bar A_{t-}$, where $N_t^{x\to y}=n$ and $N_{t-}^{x\to y}=n-1$, we let $\bar A_t=\bar A_{t-}\cup\{y\}$ if 
$$
U^{x\to y}_n\le \frac{H_{\bar A_{t-}}(\vec e)}{C \sqrt{x_2}}.
$$
\item Otherwise,  $\bar A_t=\bar A_{t-}$. 
\end{itemize}
To prove Theorem \ref{theorem: DLA_1}, we first need to show
\begin{lemma}
\label{lemma_finite_A_t}
For each time $t$, $\bar A_t$ above is with probability 1 well defined and finite. 
\end{lemma}
\begin{proof}
To prove Lemma \ref{lemma_finite_A_t}, we construct an event with probability one such that $\bar A_t$ is well defined and finite under this event. For any $x\in \mathcal{H}$ and any $0\le t<T$, define subset
\beq
\label{range_of_offsprings}
I_{t,T}(x)=\{y\in \mathcal{H}, \ \text{s.t. } (x,t)\to (y,T) \}
\eeq
and let 
$$
\mathcal{I}_{t,T}(x)=\sup_{y\in I_{t,T}(x)}\|x-y\|.
$$
Following exactly the same argument as in Lemma \ref{lemma_well_definition_2}, we have with probability one 
$$
\mathcal{I}_{0,Ct}(0)<\infty. 
$$
Under $\{\mathcal{I}_{0,Ct}(0)<\infty\}$ one can easily put all of the finite Poisson transitions within the space time box $\mathcal{I}_{0,t}(0)\times [0,t]$ in order and construct $\bar A_t$ explicitly over finite steps. Moreover, by definition we can always have $\bar A_t\subset I_{0,t}(0)$ thus $\bar A_t$ is finite. 
\end{proof}

Let $A_t=\bar A_{Ct}$, then it is easy to check $A_t$ has the same dynamic as in Theorem \ref{theorem: DLA_1} while being almost surely well defined and finite at the same time. Now, to finish the proof of Theorem \ref{theorem: DLA_1}, we again follow the argument as in Lemma \ref{lemma_well_definition_2}. 
\begin{remark}
The proof of Theorem \ref{theorem: moment} actually also contains all that is needed here (and more). Thus we will not present the details of basically the same thing for a third time. 
\end{remark}
By \eqref{probability_range_8}, \eqref{probability_range_9} and \eqref{probability_range_12} we have for any $\gamma\in (0,1/2)$ and $\ep>0$, there are constants $0<c_\gamma, C<\infty$, and deterministic sequences 
$$
t_n^{\ep,\gamma}\ge c_\gamma\ep n^{\frac{1-2\gamma}{2-2\gamma}} , \ n=1,2,\cdots
$$
and 
$$
R_n^{\ep,\gamma}\le C n^{\frac{1}{1-\gamma}} , \ n=1,2,\cdots
$$
and event $E_\infty^{\ep,\gamma}$ such that $P(E_\infty^{\ep,\gamma})\ge 1-2\ep$ and that under $E_\infty^{\ep,\gamma}$
\beq
\label{lim_sup_1}
\mathcal{I}_{0,t_n^{\ep,\gamma}}(0)\le R_n^{\ep,\gamma}, \forall \ n=1,2,\cdots.
\eeq
With \eqref{lim_sup_1} we have for any $t\ge t_1$, let 
$$
n_t=\sup\{k: \ t_k\le t\}.
$$
Then under event $E_\infty^{\ep,\gamma}$
$$
\mathcal{I}_{0,t}(0)\le R_{n_t+1}^{\ep,\gamma}\le C (n_t+1)^{\frac{1}{1-\gamma}},
$$
which implies that 
\beq
\label{lim_sup_2}
\mathcal{I}_{0,t}(0)t^{\frac{-2}{1-2\gamma}}\le 4C(c_\gamma\ep)^{\frac{-2}{1-2\gamma}}<\infty.
\eeq
Note that the right hand side of \eqref{lim_sup_2} is independent to the choice of $t$, and that $P(E_\infty^{\ep,\gamma})\ge 1-2\ep$ for all $\ep$. we have for any given $\gamma\in(0,1/2)$, with probability one
\beq
\label{lim_sup_3}
\limsup_{t\to\infty} \|A_t\|t^{\frac{-2}{1-2\gamma}}\le \limsup_{t\to\infty} \mathcal{I}_{0,Ct}(0)(Ct)^{\frac{-2}{1-2\gamma}}<\infty.
\eeq
Finally note that the choice of $\gamma$ is arbitrary and that $2/(1-\gamma)\to2$ as $\gamma\to 0$. The proof of Theorem \ref{theorem: DLA_1} is complete. \qed

\subsection{Proof of Theorem \ref{theorem: moment}}

To prove \eqref{moment}, since we have $A_t=\bar A_{Ct}\subset I_{0,Ct}(0)$, it is sufficient to show for any $t\ge 0$ and integer $m\ge 1$
\beq
\label{moment_larger}
E\left[ \mathcal{I}_{0,t}(0)^m \right]<\infty. 
\eeq
The proof here is similar to the one for Lemma \ref{lemma_well_definition_2}. However, since some more delicate estimates on the upper bounds of probabilities are needed, we still provide a detailed proof for the completeness of this paper. 

Recall \eqref{probability_range_0}, we have for any $t$ and $n$, 
\beq
\label{probability_influence_0}
\begin{aligned}
P&(\mathcal{I}_{0,t}(0)\ge n)\\
&\le P(\exists \text{ an open path in $[0,t]$ starting at $0$ with length $n$})\\
&\le \sum_{x_0, x_1,x_2,\cdots, x_n\in \mathcal{P}_{n,0}} \prod_{i=1}^n P(N^{x_{i-1}\to x_{i}}_T-N^{x_{i-1}\to x_{i}}_{T-t}\ge 1)\\
&= (1-e^{-t})\sum_{x_0, x_1,x_2,\cdots, x_n\in \mathcal{P}_{n,0}} \prod_{i=1}^{n-1} \left[1- e^{-t\sqrt{x^{(2)}_{i}}}\right]\\
&\le t^n\sum_{x_0, x_1,x_2,\cdots, x_n\in \mathcal{P}_{n,0}} \sqrt{\prod_{i=1}^{n} x^{(2)}_{i}}.
\end{aligned}
\eeq
Here we use $\mathcal{P}_{n,0}$ to denote the collection of all nearest neighbor paths starting at 0 with length $n$. Then note that $|\mathcal{P}_{n,0}|\le 4^n$ and that for each $0=x_0, x_1,x_2,\cdots, x_n\in \mathcal{P}_{n,0}$, we have 
$$
x^{(2)}_{i}\le i, \ i=0,1,2,\cdots, n. 
$$
Thus, we have 
$$
 \sqrt{\prod_{i=1}^n x^{(2)}_{i}}\le  \sqrt{n!}
$$
which implies that 
\beq
\label{probability_influence}
P(\mathcal{I}_{0,t}(0)\ge n)\le (4t)^n \sqrt{n!}.
\eeq
Now for each $\gamma\in (0,1/2)$, define 
$$
M_\gamma=\sum_{k=0}^\infty k^{2/(1-\gamma)} 2^{-k^{\gamma/(1-\gamma)}}<\infty.
$$
Now for any $\ep>0$ let 
$$
N_1=\left\lfloor\frac{4 m }{1-\gamma}\right\rfloor\ge 4m
$$ 
and 
$$
\delta_1=t_1=\frac{\ep}{64M_\gamma \sqrt{N_1}}.
$$
By \eqref{probability_influence}, we have
\beq
\label{probability_influence_1}
\begin{aligned}
P(\mathcal{I}_{0,t_1}(0)\ge N_1)&\le \left(\frac{\ep}{16M_\gamma \sqrt{N_1}} \right)^{N_1} \sqrt{N_1!}\\
&\le  \frac{\ep^{4m}}{M_\gamma} 16^{-N_1}\le   \frac{\ep^{4m}}{M_\gamma} k_1^{2/(1-\gamma)} 2^{-k_1^{\gamma/(1-\gamma)}}
\end{aligned}
\eeq
where $k_1=1$. Then under event 
$$
E_1=A_1=\{\mathcal{I}_{0,t_1}(0)< N_1\},
$$
we define 
$$
x^{(2),2}=1+N_1,
$$
$$
N_2=\left\lfloor\frac{4m \big(x^{(2),2}\big)^\gamma }{1-\gamma}\right\rfloor\ge 4m,
$$
$$
\delta_2=\frac{\ep}{64M_\gamma \sqrt{x^{(2),2}+N_2}},
$$
and
$$
t_2=t_1+\delta_2.
$$
Then define event 
$$
A_2=\bigcap_{y\in B(0,x^{(2),2}-1)}\{\mathcal{I}_{t_1,t_2}(y)< N_2\}.
$$ 
One can first see that by the same calculation as in \eqref{probability_range_1}
\beq
\label{probability_influence_2.5}
\begin{aligned}
P(A_2)&\ge 1-\sum_{y\in B(0,x^{(2),2}-1)} P(\mathcal{I}_{t_1,t_2}(y)\ge N_2)\\
&\ge 1-4\big(x^{(2),2}\big)^{2}\left(\frac{\ep}{16M_\gamma \sqrt{x^{(2),2}+N_2}}\right)^{N_2}\sqrt{\prod_{j=x^{(2),2}-1}^{x^{(2),2}+N_2-2}j}\\
&\ge 1-\big(x^{(2),2}\big)^{2}\frac{\ep^{4m}}{16M_\gamma}2^{-k_2^{\gamma/(1-\gamma)}}
\end{aligned}
\eeq
where
$$
k_2=\lfloor\big(x^{(2),2}\big)^{1-\gamma}\rfloor.
$$
The last inequality in \eqref{probability_influence_2.5} results from
$$
N_2=\left\lfloor\frac{4m \big(x^{(2),2}\big)^\gamma }{1-\gamma}\right\rfloor\ge \big(x^{(2),2}\big)^\gamma\ge k_1^{\gamma/(1-\gamma)}. 
$$

Moreover, we have 
$$
\big(x^{(2),2}\big)^{2}=\left[(x^{(2),2}\big)^{1-\gamma} \right]^{2/(1-\gamma)}
$$
while
$$
(x^{(2),2}\big)^{1-\gamma} \le 2k_2.
$$
Thus
\beq
\label{probability_influence_3}
\begin{aligned}
P(A_2)&\ge1-\big(2k_2\big)^{2/(1-\gamma)}\frac{\ep^{4m}}{16M_\gamma}2^{-k_2^{\gamma/(1-\gamma)}}\\
&\ge 1-\frac{\ep^{4m}}{M_\gamma}k_2^{2/(1-\gamma)}2^{-k_2^{\gamma/(1-\gamma)}}.
\end{aligned}
\eeq
Then note that for any $x\ge 1$, we have by calculus 
$$
\big(x^{1-\gamma}+4m\big)^{1/(1-\gamma)}> x+\frac{4m}{1-\gamma} x^\gamma
$$
while 
$$
\begin{aligned}
\big(x^{1-\gamma}+1\big)^{1/(1-\gamma)}&\le x+\frac{1}{1-\gamma} \big(x^{1-\gamma}+1\big)^{\gamma/(1-\gamma)}\\
&\le x+\frac{1}{1-\gamma} \big(2x^{1-\gamma}\big)^{\gamma/(1-\gamma)}\\
&< x+\frac{2}{1-\gamma} x^\gamma\\
&< x+\left\lfloor\frac{4m x^\gamma }{1-\gamma}\right\rfloor.
\end{aligned}
$$
We have that 
\beq
\label{probability_influence_3.5}
\big(x^{(2),2}\big)^{1-\gamma}=\big(1+N_1\big)^{1-\gamma}=\left(1+\left\lfloor\frac{4m\cdot 1^\gamma }{1-\gamma}\right\rfloor \right)^{1-\gamma}\in \big(2,1+4m \big)
\eeq
and that
$$
k_2=\lfloor\big(x^{(2),2}\big)^{1-\gamma}\rfloor\in (k_1+1,k_1+4m) 
$$
since $k_1=1$. Then using exactly the same argument on 
$$
x^{(2),3}=x^{(2),2}+N_2,
$$
and 
$$
k_3=\lfloor\big(x^{(2),3}\big)^{1-\gamma}\rfloor,
$$
we have 
\beq
\label{probability_influence_4}
\begin{aligned}
\big(x^{(2),3}\big)^{1-\gamma}=\big(x^{(2),2}+N_2\big)^{1-\gamma}&=\left(x^{(2),2}+\left\lfloor\frac{4m \big(x^{(2),2}\big)^\gamma }{1-\gamma}\right\rfloor \right)^{1-\gamma}\\
&\in \big((x^{(2),2})^{1-\gamma}+1,(x^{(2),2})^{1-\gamma}+4m \big)
\end{aligned}
\eeq
and thus
$$
k_3\in (k_2+1,k_2+4m). 
$$
Then we note that the event $A_1$ depends only on the transitions within $B(0,N_1)\times [0,t_1]$, while event $A_2$ depends only on the transitions within $B(0,N_1+N_2)\times [t_1,t_2]$. By the independence of Poisson process increments, we have that $A_1$ is independent to $A_2$, and thus for $E_2=A_1\cap A_2$,
$$
P(E_2)=P(A_1)P(A_2)\ge \left( 1-\frac{\ep^{4m}}{M_\gamma}k_1^{2/(1-\gamma)}2^{-k_1^{\gamma/(1-\gamma)}}\right)\cdot\left( 1-\frac{\ep^{4m}}{M_\gamma}k_2^{2/(1-\gamma)}2^{-k_2^{\gamma/(1-\gamma)}}\right).
$$
Finally, recalling the definition of $\mathcal{I}_{0,t}$, one can immediately have under $E_2$
$$
\mathcal{I}_{0,t_2}(0)<x^{(2),3}<\infty. 
$$
Repeat the iteration above, i.e., for all $n\ge 3$ let 
$$
x^{(2),n}=x^{(2),n-1}+N_{n-1},
$$
$$
N_n=\left\lfloor\frac{4m \big(x^{(2),n}\big)^\gamma }{1-\gamma}\right\rfloor,
$$
$$
\delta_n=\frac{\ep}{64M_\gamma \sqrt{x^{(2),n}+N_n}},
$$
$$
t_n=t_{n-1}+\delta_n,
$$
$$
A_n=\bigcap_{y\in B(0,x^{(2),n}-1)}\{\mathcal{I}_{t_{n-1},t_n}(y)< N_n\},
$$ 
and
$$
E_n=E_{n-1}\cap A_n.
$$
Consider 
$$
E^{\ep}_\infty=\bigcap_{n=1}^\infty A_n.
$$
Under $E^{\ep}_\infty$ we have for any $n\ge1$, 
\beq
\label{probability_influence_5}
\mathcal{I}_{0,t_n}(x)<x^{(2),n+1}<\infty. 
\eeq
At the same time
\beq
\label{probability_influence_6}
\begin{aligned}
t_n&=\sum_{i=1}^n\delta_i=\sum_{i=1}^n \frac{\ep}{64M_\gamma \sqrt{x^{(2),i}+N_i}}=\frac{\ep}{64M_\gamma}\sum_{i=1}^n \frac{1}{ \sqrt{x^{(2),i+1}}}.
\end{aligned}
\eeq
Moreover, by \eqref{probability_influence_4} we have for each $i$
\beq
\label{probability_influence_7}
\big(x^{(2),i}\big)^{1-\gamma}=\big(x^{(2),i-1}+N_{i-1}\big)^{1-\gamma}\in \big(\big(x^{(2),i-1}\big)^{1-\gamma}+1,\big(x^{(2),i-1}\big)^{1-\gamma}+4m \big)
\eeq
which together implies that 
\beq
\label{probability_influence_8}
\big(x^{(2),i}\big)^{1-\gamma}\le 4im. 
\eeq
Combining \eqref{probability_influence_6}-\eqref{probability_influence_8} we have 
\beq
\label{probability_influence_9}
\begin{aligned}
t_n\ge\frac{\ep}{64M_\gamma}\sum_{i=1}^n \left(4im\right)^{-1/(2-2\gamma)}.
\end{aligned}
\eeq

Recalling that $\gamma\in (0,1/2)$, $1/(2-2\gamma)<1$, which implies that the series in \eqref{probability_influence_9} is divergent. So for any $t\ge 0$ there is an $n_0<\infty$ such that for all $n\ge n_0$, $t_n\ge t$, and that $t_{n_0-1}<t$. 
$$
\mathcal{I}_{0,t}(0)\le\mathcal{I}_{0,t_{n_0}}(0).
$$
And by \eqref{probability_influence_5} and \eqref{probability_influence_8}, under $E^{\ep}_\infty$,
$$
\mathcal{I}_{0,t_{n_0}}(0)<x^{(2),n+1}\le [4m(n_0+1)]^{1/(1-\gamma)}. 
$$
Thus we have under event $E^{\ep}_\infty$,
\beq
\label{probability_influence_9.5}
\mathcal{I}_{0,t}(0)\le [4m(n_0+1)]^{1/(1-\gamma)}\le  (8m)^{1/(1-\gamma)}\cdot n_0^{1/(1-\gamma)}
\eeq
On the other hand, 
\beq
\label{probability_influence_10}
t\ge t_{n_0-1}\ge \frac{\ep}{64M_\gamma}\sum_{i=1}^{n_0-1} \left(4m\cdot i\right)^{-1/(2-2\gamma)}\ge \frac{c\ep}{64M_\gamma (4m)^{1/(2-2\gamma)}} n_0^{(1-2\gamma)/(2-2\gamma)}. 
\eeq
Combining \eqref{probability_influence_9.5} and \eqref{probability_influence_10}, we have under event $E^{\ep}_\infty$ there is a constant $C_{m,\gamma}$ depending on $m$ and $\gamma$ but independent to $\ep$ such that   
\beq
\label{probability_influence_10.5}
\mathcal{I}_{0,t}(0)\le C_{m,\gamma} \ep^{-2/(1-\gamma)}t^{2/(1-2\gamma)}. 
\eeq
Note that by the independence of Poisson processes increments, we have that $A_1, A_2,\cdots$ gives a sequence of independent events. And according to \eqref{probability_influence_2.5} and the construction in our iteration, we have for each $i$
\beq
\label{probability_influence_11}
\begin{aligned}
P(A_i)\ge 1-\frac{\ep^{4m}}{M_\gamma}k_i^{2/(1-\gamma)}2^{-k_i^{\gamma/(1-\gamma)}}>1-\ep
\end{aligned}
\eeq
with 
\beq
\label{probability_influence_11.5}
k_i=\lfloor\big(x^{(2),i}\big)^{1-\gamma} \rfloor\in [k_{i-1}+1,k_{i-1}+4m). 
\eeq
Thus for sufficiently small $\ep$ such that for all $x\in (0,\ep)$, $\log(1-x)\ge -2x$ and any $n\ge 1$, we have
$$
P(E_n)-1\ge\log\big( P(E_n)\big)=\log\left( \prod_{i=1}^n P(A_i)\right)\ge 2\sum_{i=1}^n[P(A_i)-1]. 
$$
By \eqref{probability_influence_11}, 
$$
P(E_n)-1\ge -\frac{2\ep^{4m}}{M_\gamma}\sum_{i=1}^n k_i^{2/(1-\gamma)}2^{-k_i^{\gamma/(1-\gamma)}}.
$$
Then noting that by \eqref{probability_influence_11.5} $k_i\ge k_{i-1}+1$ and the fact that all $k_i$'s are integers by definition, we have
$$
\sum_{i=1}^n k_i^{2/(1-\gamma)}2^{-k_i^{\gamma/(1-\gamma)}}\le M_\gamma
$$
and thus
\beq
\label{probability_influence_13}
P(E_n)\ge 1-2\ep^{4m}. 
\eeq
Note that the right hand side of \eqref{probability_influence_13} is independent of $n$. We have $P(E^\ep_\infty)\ge 1-2\ep^{4m}$. Now let 
$$
\ep_j=\left(\frac{1}{j}\right)^{\frac{1-r}{2m}}. 
$$
Then we have for each $j$ sufficiently large, 
\beq
\label{probability_influence_14}
\begin{aligned}
P\left(\frac{\mathcal{I}_{0,t}(0)^m}{(C_{m,\gamma})^mt^{2m/(1-2\gamma)}}> \ep_j^{-2m/(1-\gamma)}\right)&=P\left(\mathcal{I}_{0,t}(0)> C_{m,\gamma} \ep_j^{-2/(1-\gamma)}t^{2/(1-2\gamma)}\right)\\
&\le 1-P(E_\infty^{\ep_j})\le 2\ep_j^{4m}
\end{aligned}
\eeq
and thus 
\beq
\label{probability_influence_15}
\begin{aligned}
P\left(\frac{\mathcal{I}_{0,t}(0)^m}{(C_{m,\gamma})^mt^{2m/(1-2\gamma)}}> j\right)\le 2\left(\frac{1}{j}\right)^{2(1-\gamma)}.
\end{aligned}
\eeq
Noting that $\gamma<1/2$ and thus $2(1-\gamma)>1$, 
$$
\sum_{n=1}^\infty P\left(\frac{\mathcal{I}_{0,t}(0)^m}{(C_{m,\gamma})^mt^{2m/(1-2\gamma)}}> j\right)<\infty
$$
which implies that $E[\mathcal{I}_{0,t}(0)^m]<\infty$. \qed

\section*{Acknowledgments} 
We would like to thank Itai Benjamini, Noam Berger, Marek Biskup, Rick Durrett and Gady Kozma for fruitful discussions in the initiation of this project. 

\bibliography{career}
\bibliographystyle{plain}

\end{document}

%% file: outline.tex
\begin{tikzpicture}[scale=1.0]
\tikzstyle{redcirc}=[circle,
draw=black,fill=myred,thin,inner sep=0pt,minimum size=2mm]
\tikzstyle{bluecirc}=[circle,
draw=black,fill=blue,thin,inner sep=0pt,minimum size=2mm]

\node (v1) at (0,2) [bluecirc] {};
\node (v2) at (5.5,0) {$L_0$};
\node (v3) at (5.5,2) {$L_n$};
\node (v4) at (5.5,4) {$L_{2n}$};
\node (v5) at (5.5,8) {$L_N$};
\node (v6) at (0.5,1) {$V_n$};

\draw [thick] (-5,0) to (5,0);
\draw [thick,dashed] (-5,2) to (5,2);
\draw [thick] (-5,4) to (5,4);
\draw [thick] (-5,8) to (5,8);
\draw [thick,blue] (0,0) to (0,2);
\draw (-1,2) to (0,3);
\draw (1,2) to (0,3);
\draw [dashed](-1,2) to (0,1);
\draw [dashed] (1,2) to (0,1);

\draw [red,thick] plot [smooth, tension=2.5] coordinates { (0,2) (-0.1,2.2) (0.5,2.5)};
\draw [green,thick] plot [smooth, tension=2.5] coordinates {(0.5,2.5) (0.7,2.6)(1,2.2) (1.3,3) (2,3.2) (1.5,2.7) (0.5,4) };
\draw [myblue,thick] plot [smooth, tension=2.5] coordinates {(0.5,4)(2,6)(1,5)(-2,3.5)(-2,8)};

\end{tikzpicture}